\documentclass[11pt,a4paper]{article}
\usepackage{raguetsty/raguet_en}

\title{Generalized Forward-Backward Splitting}

\author{
Hugo Raguet$^1$ \quad Jalal Fadili$^2$ \quad Gabriel Peyr\'e$^1$\\
\begin{tabular}{c}
\\
\\
$^1$Ceremade\\
CNRS-Universit\'e Paris-Dauphine\\
Pl. De Lattre De Tassigny,\\
75775 Paris Cedex 16\\
France\\
{\small \url{{raguet,peyre}@ceremade.dauphine.fr}}\\
\end{tabular}
\begin{tabular}{c}
\\
\\
$^2$GREYC\\
CNRS-ENSICAEN-Universit\'e de Caen\\
6, Bd du Mar\'echal Juin,\\
14050 Caen Cedex\\
France\\
{\small \url{Jalal.Fadili@greyc.ensicaen.fr}}\\
\end{tabular}
}

\begin{document}

\maketitle

\begin{abstract}
This paper introduces the generalized forward-backward splitting algorithm for minimizing convex functions of the form $F + \sum_{i=1}^n G_i$, where $F$ has a Lipschitz-continuous gradient and the $G_i$'s are simple in the sense that their Moreau proximity operators are easy to compute. While the forward-backward algorithm cannot deal with more than $n = 1$ non-smooth function, our method generalizes it to the case of arbitrary $n$. Our method makes an explicit use of the regularity of $F$ in the forward step, and the proximity operators of the $G_i$'s are applied in parallel in the backward step. This allows the generalized forward-backward to efficiently address an important class of convex problems. We prove its convergence in infinite dimension, and its robustness to errors on the computation of the proximity operators and of the gradient of $F$. Examples on inverse problems in imaging demonstrate the advantage of the proposed methods in comparison to other splitting algorithms.
\end{abstract}

\newcommand{\Ademi}{\Aa\pa{ \frac{1}{2} }}

\newcommand{\gainf}{\ga_{\!\infty}}
\newcommand{\bgainf}{\bga_{\!\infty}}
\newcommand{\bgaA}{\bga \! \bA}
\newcommand{\bgatA}{{\bga}_{t} \! \bA}
\newcommand{\bgainfA}{{\bgainf} \! \bA}

\newcommand{\JNS}{J_{N_{\bSs}}}
\newcommand{\RNS}{R_{N_{\bSs}}}
\newcommand{\JgA}{J_{\bgaA}}
\newcommand{\RgA}{R_{\bgaA}}
\newcommand{\JgAi}{J_{\frac{\ga}{\om_i} A_i}}
\newcommand{\RgAi}{R_{\frac{\ga}{\om_i} A_i}}
\newcommand{\JgiAi}{J_{\ga_i A_i}}
\newcommand{\RgiAi}{R_{\ga_i A_i}}

\newcommand{\Tb}{\bT_{1,\bga}}
\newcommand{\Tf}{\bT_{2,\ga}}
\newcommand{\Tbt}{\bT_{1,\bga_t}}
\newcommand{\Tft}{\bT_{2,\ga_t}}
\newcommand{\Tbb}{\bT_{1,\bar{\bga}}}
\newcommand{\Tfb}{\bT_{2,\bar{\ga}}}
\newcommand{\Tbinf}{\bT_{1,\bgainf}}
\newcommand{\Tfinf}{\bT_{2,\gainf}}
\newcommand{\bBJ}{\bB\JNS}
\newcommand{\Apga}{\bA_{\bga}\prim}
\newcommand{\Apgat}{\bA_{\bga_t}\prim}
\newcommand{\Apgainf}{\bA_{\bgainf}\prim}
\newcommand{\projS}{\proj_{\bSs}}
\newcommand{\projP}{\proj_{\bSs^\bot}}
\newcommand{\pSs}[1]{{#1}^{\bSs}}
\newcommand{\pSp}[1]{{#1}^{\boldsymbol \bot}}
\newcommand{\Fx}{\mathbf{F}}

\newcommand{\musp}{\mu}
\newcommand{\mutv}{\nu}
\newcommand{\normbk}[1]{\norm{#1}_{1,2}^{\Bb}}
\newcommand{\normtv}[1]{\norm{#1}_{\mathrm{TV}}}
\newcommand{\normbki}[1]{\norm{#1}_{1,2}^{\Bb_i}}
\newcommand{\normbktv}[1]{\norm{#1}_{1,2}^{\Bb_{\mathrm{TV}}}}

\newcommand{\RNN}{\Ii}
\newcommand{\RNJ}{\Ii^J}
\newcommand{\RG}{\Ii^2}
\newcommand{\Ks}{K}
\newcommand{\Mr}{M}

\newcommand{\gradient}{\nabla_{\Ii}}

\newcommand{\GFB}{\textsf{GFB}\xspace}
\newcommand{\DR}{\textsf{DR}\xspace}
\newcommand{\ChPo}{\textsf{ChPo}\xspace}
\newcommand{\CoPe}{\textsf{CoPe}\xspace}
\newcommand{\HPE}{\textsf{HPE}\xspace}

\section{Introduction}

Throughout this paper, $\Hh$ denotes a real Hilbert space endowed with scalar product $\dprod{\cdot}{\cdot}$ and associated norm $\norm{\cdot}$, and $n$ is a positive integer. We consider the following minimization problem
\begin{equation}\label{minFnG}
 \umin{x \in \Hh} \{\Psi(x) \eqdef F(x)+\sum_{i=1}^{n} G_i(x)\},
\end{equation}
where all considered functions belong to the class $\Ga_0(\Hh)$ of lower semicontinuous, proper (its domain is non-empty) and convex functions from $\Hh$ to $]\minfty,\pinfty]$.

\subsection{State-of-the-Art in Splitting Methods}

The decomposition \eqref{minFnG} is fairly general, and a wide range of iterative algorithms takes advantage of the specific properties of the functions in the summand. One crucial property is the possibility to compute the associated  proximity operators \cite{Moreau65}, defined as
\begin{equation}\label{proxdef}
  \prox_{G}(x) \eqdef \uargmin{y \in \Hh} \frac{1}{2} \norm{x - y}^2 + G(y).
\end{equation}
This is in itself a convex optimization problem, which can be solved efficiently for many functions, \eg when the solution, unique by strong convexity, can be written in closed form. Such functions are referred to as ``simple''.

Another important feature is the differentiability of the functional to be minimized. However, gradient-descent approaches do not apply as soon as one of the functions $G_i$ is non-smooth. For ${n \equiv 1}$ and $G_1$ simple, the forward-backward algorithm circumvents this difficulty if $F$ is differentiable with a Lipschitz-continuous gradient. This scheme consists in performing alternatively a gradient-descent (corresponding to an explicit step on the function $F$) followed by a proximal step (corresponding to an implicit step on the function $G_1$). Such a scheme can be understood as a generalization of the projected gradient method. This algorithm has been well studied~\cite{Mercier79,Gabay83,Tseng91,ChenRockafellar97,Tseng00,CombettesWajs05,Bredies08}. Accelerated multistep versions have been proposed~\cite{Nesterov07,Tseng08,BeckTeboulle09}, that enjoy a faster convergence rate of $O(1/t^2)$ on the objective $\Psi$.

Other splitting methods do not require any smoothness on some part of the composite functional $\Psi$. The Douglas-Rachford~\cite{DouglasRachford56} and Peaceman-Rachford~\cite{PeacemanRachford55} schemes were developed to minimize $G_1(x)+G_2(x)$, provided that $G_1$ and $G_2$ are simple \cite{LionsMercier79,Lieutaud69,EcksteinBertsekas92,Combettes04} and rely only on the use of proximity operators. The backward-backward algorithm~\cite{Lions78,Passty79,AckerPrestel80,Bauschke05,Combettes04} can be used to minimize $\Psi(x) = G_1(x)+G_2(x)$ when the functions involved are the indicator functions of non-empty closed convex sets, or involve Moreau envelopes.  Interestingly, if one of the functions $G_1$ or $G_2$ is a Moreau envelope and the other is simple, the forward-backward algorithm amounts to a backward-backward scheme.

If $L$ is a bounded injective linear operator, it is possible to minimize $\Psi(x) = G_1 \circ L(x) + G_2(x)$ by applying these splitting schemes on the Fenchel-Rockafellar dual problem. It was shown that applying the Douglas-Rachford scheme leads to the alternating direction method of multipliers (ADMM)~\cite{FortinGlowinski83, Gabay83,GabayMercier76,GlowinskiLeTallec89,EcksteinBertsekas92}. For non-necessarily injective $L$ and $G_2$ strongly convex with a Lipschitz-continuous gradient, the forward-backward algorithm can be applied to the Fenchel-Rockafellar dual \cite{FadiliPeyre10,CombettesDualization10}. Dealing with an arbitrary bounded linear operator $L$ can be achieved using primal-dual methods motivated by the classical Kuhn-Tucker theory. Starting from methods to solve saddle function problems such as the Arrow-Hurwicz  method \cite{ArrowHurwicz58} and its modification \cite{Popov80}, the extragradient method \cite{Korpelevich76}, this problem has received a lot of attention more recently \cite{ChenTeboulle94,Tseng97,Solodov04,MonteiroSvaiter10,ChambollePock11,Briceno-AriasCombettes11}.

It is also possible to extend the Douglas-Rachford algorithm to an arbitrary number $n > 2$ of simple functions. Inspired by the method of partial inverses \cite[Section~5]{Spingarn83}, most methods rely either explicitly or implicitly on introducing auxiliary variables and bringing back the original problem to the case $n = 2$ in the product space $\Hh^{n}$. Doing so yields iterative schemes in which one performs independent parallel proximal steps on each of the simple functions and then computes the next iterate by essentially averaging the results. Variants have been proposed in \cite{CombettesPesquet08} and \cite{EcksteinSvaiter09}, who describe a general projective framework that does not reduce the problem to the case $n = 2$. Note however that these extensions do not apply to the forward-backward scheme that can only handle $n \equiv 1$. It is at the heart of this paper to present such an extension. 

Recently proposed methods extend existing splitting schemes to handle the sum of any number of $n \geq 2$ composite functions of the form $G_i = H_i \circ L_i$, where the $H_i$'s are simple and the $L_i$'s are bounded linear operators. Let us denote $\adj{L_i}$ the adjoint operator of $L_i$. If $L_i$ satisfies $L_i  \adj{L_i} = \nu \Id$ for any $\nu > 0$ (it is a so-called \textit{tight frame}), $H_i \circ L_i$ is simple as soon as $H_i$ is simple and $\adj{L_i}$ is easy to compute \cite{Combettes2007a}. This case thus reduces to the previously reviewed ones. If $L_i$ is not a tight frame but $\pa{\Id + \adj{L_i} L_i}$ or $\pa{\Id + L_i \adj{L_i}}$ is easily invertible, it is again possible to reduce the problem to the previous cases by introducing as many auxiliary variables as the number of $L_i$'s each belonging to the range of $L_i$. Note however that, if solved with the Douglas-Rachford algorithm on the product space, the auxiliary variables are also duplicated, which would increase significantly the dimensionality of the problem. Some dedicated parallel implementations were specifically designed for the case where $\pa{\sum_i \adj{L_i} L_i}$ or $\pa{\sum_i L_i \adj{L_i}}$ is (easily) invertible, see for instance \cite{Eckstein94,PesquetPustelnik11}. If the $L_i$'s satisfy none of the above properties, it is still possible to call on primal-dual methods, either by writing $\Psi(x) = \sum_{i=1}^n H_i(L_ix) = G  (Lx)$ with $L(x) = \pa{L_i(x)}_i$ and $G\bpa{\pa{x_i}_i} = \sum_i H_i(x_i)$, see for instance~\cite{Dupe11b}; or $\Psi(\pa{x_i}_i) = \iota_{\bSs}(\pa{x_i}_i) + \sum_i H_i(L_ix_i)$ \cite{Briceno-AriasCombettes11}, where $\bSs$ is the closed convex set defined in \sref{subsec-defn-prop-other}.

In spite of the wide range of already existing proximal splitting methods, none seems satisfying to address explicitly the case where $n > 1$ and $F$ is smooth but not necessarily simple. A workaround that has been proposed previously used nested algorithms to compute the proximity operator of $\sum_i G_i$ within sub-iterations, see for instance~\cite{Dupe2009,Chaux09}; this leads to practical as well as theoretical difficulties to select the number of sub-iterations. More recently, \cite{MonteiroSvaiter10} proposed an algorithm for minimizing $\Psi(x) = F(x)+G(x)$ under linear constraints. We show in \sref{discussion} how this can be adapted to adress the general problem \eqref{minFnG} while achieving full splitting of the proximity operators of the $G_i$'s and using the gradient of $F$. It suffers however from limitations, in particular the introduction of many auxiliary variables and the fact that the gradient descent can't be directly applied to the minimizer; see \sref{discussion} and \ref{numeric} for details. The generalized forward-backward algorithm introduced in this paper is intended to avoid all those shortcomings.

As this paper was being finalized, the authors in \cite{CombettesPesquet11} independently developed a primal-dual algorithm to solve a class of problems that cover those we consider here. Their approach and algorithm are however very different from ours in many important ways. We will provide a detailed comparison with this work in \sref{discussion} and will also show on numerical experiments in \sref{numeric} that our algorithm seems more adapted for problems of the form \eqref{minFnG}.

\subsection{Applications in Image Processing}

Many imaging applications require solving ill-posed inverse problems to recover high quality images from low-dimensional and noisy observations. These challenging problems necessitate the use of regularization through prior models to capture the geometry of natural signals, images or videos. The resolution of the inverse problem can be achieved by minimizing objective functionals, with respect to a high-dimensional variable, that takes into account both a fidelity term to the observations and regularization terms reflecting the priors. Clearly, such functionals are composite by construction. \sref{numeric} details several examples of such inverse problems.  

In many situations, this leads to the optimization of a convex functional that can be split into the sum of convex smooth and non-smooth terms. The smooth part of the objective is in some cases a data fidelity term and reflects some specific knowledge about the forward model, \ie the noise and the measurement/degradation operator. This is for instance the case if the operator is linear and the noise is additive Gaussian, in which case the data fidelity is a quadratic function. The most successful regularizations that have been advocated are non-smooth, which typically allow to preserve sharp and intricate structures in the recovered data. Among such priors, sparsity-promoting ones have become popular, \eg the $\ell_1$-norm of coefficients in a wisely chosen dictionary \cite{Mallat99}, or total variation (TV) prior~\cite{ROF92}. To better model the data, composite priors can be constructed by summing several suitable regularizations, see for instance the morphological diversity framework~\cite{FadiliStarckBook09}. The proximity operator of the $\ell_1$-norm penalization is a simple soft-thresholding~\cite{Donoho95}, whereas the use of complex or mixed regularization priors justifies the splitting of non-smooth terms in several simpler functions (see \sref{numeric} for concrete examples).

The composite structure of convex optimization problems raising when solving inverse problems in the form of a sum of simple and/or smooth functions involving linear operators explains the popularity of proximal splitting schemes in imaging science. Depending on the structure of the objective functional as detailed in the previous section, one can resort to the appropriate splitting algorithm. For instance, the forward-backward algorithm and its modifications has become popular for sparse regularization with a smooth data fidelity, see for instance~\cite{wave:nowak03,DaubechiesDemol04,CombettesWajs05,Fadili06,chauxframes,BeckTeboulle09,Briceno-Arias09}. The Douglas-Rachford and its parallelized extensions were also used in a variety of inverse problems implying only non-smooth functions, see for instance \cite{Combettes2007a,CombettesPesquet08,Dupe2009,Chaux09,Briceno-Arias10,Dupe11a,Dupe11c,Pustelnik11}. The ADMM (which is nothing but Douglas-Rachford on the dual) was also applied to some linear inverse problems in \cite{Figueiredo10a,Figueiredo10b}. Primal-dual schemes \cite{ChambollePock11,Dupe11b} are among the most flexible schemes to handle more complicated priors. The interested reader may refer to \cite[Chapter 7]{FadiliStarckBook09} and \cite{CombettesPesquet09} for extensive reviews.

\subsection{Contributions and Paper Organization}

This paper introduces a novel generalized forward-backward algorithm to solve \eqref{minFnG} when $F$ is convex with a Lipschitz continuous gradient, and the $G_i$'s are convex and simple. The algorithm achieves full splitting where all operators are used separately: an explicit step for $\nabla F$ (single-valued) and a parallelized implicit step through the proximity operators of the $G_i$'s. We prove convergence of the algorithm as well as its robustness to errors that may contaminate the iterations. To the best of our knowledge, it is among the first algorithms to tackle the case where $n > 1$ and $F$ is smooth (see \sref{discussion} for relation to a recent work developed in parallel to ours). Although our numerical results are reported only on imaging applications, the algorithm may prove useful for many other applications such as machine learning or statistical estimation.

\sref{algorithm} presents the algorithm and state our main theoretical result. \sref{preliminaries}, that can be skipped by experienced readers, sets some necessary material from the framework of monotone operator theory. \sref{proof} reformulates the generalized forward-backward algorithm for finding the zeros of the sum of maximal monotone operators, and proves its convergence and its robustness. Special instances of the algorithm, its potential extensions and discussion of its relation to two alternatives in the literature are given in \sref{discussion}. Numerical examples are reported in \sref{numeric} to show the usefulness of this approach for applications to imaging problems.

\section{Generalized For\-ward-Back\-ward Al\-go\-rithm\\ for Mi\-ni\-mi\-za\-tion Problems}
\label{algorithm}

We consider problem \eqref{minFnG} where all functions are in $\Ga_0(\Hh),$ $F$ is differentiable on $\Hh$ with $1/\beta$-Lipschitz gradient where $\beta \in ]0,\pinfty[$, and for all $i$, $G_i$ is simple. We also assume the following:
\begin{enumerate}[label={\rm (H\arabic{*})}, ref={\rm (H\arabic{*})}]
\item \label{H:argmin} The set of minimizers of \eqref{minFnG} $\argmin(\Psi)$ is non-empty;
\item \label{H:sri} The domain qualification condition holds, \ie
\[
(0,\ldots,0) \in \sri\{(x-y_1,\ldots,x-y_n) ~ \big | x \in \Hh \sobjs{and} \foralls i, \, y_i \in \dom(G_i)\}~,
\]
\end{enumerate}
where $\dom(G_i) \eqdef \{x \in \Hh \big | G_i(x) < \pinfty\}$ and $\sri(C)$ is the strong relative interior of a non-empty convex subset $C$ of $\Hh$ \cite{BauschkeCombettes11}.
Under \ref{H:argmin}-\ref{H:sri}, it follows from \cite[Theorem~16.2 and Theorem~16.37(i)]{Rockafellar70,BauschkeCombettes11} that
\[
\emptyset \neq \argmin(\Psi) = \zer(\partial \Psi) = \zer \left( \nabla F + \textsum{i}{} \partial G_i \right)~,
\]
where $\partial G_i$ denotes the subdifferential of $G_i$ and $\zer(A)$ is the set of zeros of a set-valued map $A$ (see \dref{def:setvalued} in \sref{subsec-defn-prop}).
Therefore, solving \eqref{minFnG} is equivalent to
\begin{equation}\label{cns}
\sobjs{Find} x \in \Hh \sobjs{such that} 0 \in \nabla F(x) + \sum_i \partial G_i(x) ~.
\end{equation}

The generalized forward-backward we propose to minimize \eqref{minFnG} (or equi\-valently to solve \eqref{cns}) is detailed in \aref{gfb-algo}. 

\begin{algorithm}[H]
\caption{Generalized Forward-Backward Algorithm for solving \eqref{minFnG}.\newline
$\inv{\be} \in ]0,\pinfty[$ is the Lipschitz constant of $\nabla F$; $I_\la$ is defined in \tref{gfb-algo-thm}.}
{\noindent{\bf{Require}}}
 $\begin{array}{ll}
	 {\pa{z_i}_{i \in \bbket{1,n}} \in \, \Hh^n}, & {\pa{\om_i}_{i \in \, \bbket{1,n}} \in {]0,1[}^n \sobjs{s.t.} \sum_{i=1}^n \om_i = 1}, \\
	 \ga_t \in \, ]0,2\be[ ~ \forall t \in \N, & \la_t \in I_\la ~ \forall t \in \N ~. \end{array}$\\
{\noindent{\bf{Initialization}}}\\
$x \leftarrow \sum_i \om_i z_i$;\\
$t \leftarrow 0$.\\
{\noindent{\bf{Main iteration}}}\\
\Repeat{convergence}{
	\For{$i \in \bbket{1,n}$}{
		 $\displaystyle z_i \leftarrow z_i + \la_t \bpa{\prox_{\frac{\ga_t}{\om_i} G_i} \bpa{ 2x - z_i - \ga_t \nabla F(x) } - x}$;
	}
	$x \leftarrow \sum_i \om_i z_i$;\\
	$t \leftarrow t+1$.
}
{\noindent{\bf{Return}}} $x$.
\label{gfb-algo}
\end{algorithm}

To state our main theorem that ensures the convergence of the algorithm and its robustness, for each $i$ let $\Eps_{1,t,i}$ be the error at iteration $t$ when computing $\prox_{\frac{\ga_t}{\om_i} G_i}$ at its argument, and let $\Eps_{2,t}$ be the error at iteration $t$ when applying $\nabla F$ to its argument. \aref{gfb-algo} generates sequences $\pa{z_{i,t}}_{t \in \N}$, $i \in \bbket{1,n}$ and $\pa{x_t}_{t \in \N}$, such that for all $i$ and $t$,
\begin{equation}\label{seq-w-err}
	z_{i,t+1} = z_{i,t} + \la_t \bpa{\prox_{\frac{\ga_t}{\om_i} G_i} \bpa{ 2 x_t - z_{i,t} - 
	\ga_t \pa{ \nabla F(x_t) + \Eps_{2,t} } } + \Eps_{1,t,i} - x_t } ~.
\end{equation}

The following theorem introduces two different sets of assumptions to guarantee convergence. Assumption~\ref{big-la-algo} allows one to use a greater range for the relaxation parameters $\la_t$, while assumptions \ref{var-ga-algo} enables varying gradient-descent step-size $\ga_t$ and ensures strong convergence in the uniformly convex case. Recall that a function $F \in \Ga_0(\Hh)$ is uniformly convex if there exists a non-decreasing function $\phy: [0,\pinfty[ \to [0,\pinfty]$ that vanishes only at 0, such that for all $x$ and $y$ in $\dom(F)$, the following holds
\begin{equation}
\foralls \rho \in ]0,1[, \, F(\rho x + (1-\rho) y) + \rho(1-\rho)\phy(\norm{x-y}) \leq \rho F(x) + (1-\rho)F(y).
\end{equation}

\begin{theorem}
Set $\varlimsup \ga_t = \bar{\ga}$ and define the following assumptions:
\begin{enumerate}[label={\rm (A\arabic{*})}, ref={\rm (A\arabic{*})}]
\setcounter{enumi}{-1}
\item{\label{lim-la-sum-err-algo}}
	\begin{enumerate}[label={\rm (\roman{*})}, ref={\rm (\roman{*})}]
		\item{\label{lim-la-algo}} $0 < \varliminf \la_t \leq \varlimsup \la_t < \min\pa{\frac{3}{2},\frac{1+2\be/\bar{\ga}}{2}}$;
		\item{\label{sum-err-algo}} $\sum_{t=0}^{\pinfty} \norm{\Eps_{2,t}} < \pinfty$, and for all $i$, $\sum_{t=0}^{\pinfty} \norm{\Eps_{1,t,i}} < \pinfty$.
	\end{enumerate} 
\item{\label{big-la-algo}}
	\begin{enumerate}[label={\rm (\roman{*})}, ref={\rm (\roman{*})}]
		\item{\label{fix-ga-algo}} $\foralls t, \, \ga_t = \bar{\ga} \in ]0,2\be[$;
		\item{\label{bound-la-algo}} $I_\la = \left]0,\min\pa{\frac{3}{2},\frac{1+2\be/\bar{\ga}}{2}}\right[$.
	\end{enumerate} 
\item{\label{var-ga-algo}}
	\begin{enumerate}[label={\rm (\roman{*})}, ref={\rm (\roman{*})}]
		\item{\label{lim-ga-algo}} $0 < \varliminf \ga_t \leq \bar{\ga} < 2\be$;
		\item{\label{small-la-algo}} $I_\la = ]0,1]$.
	\end{enumerate}
\end{enumerate} 
Suppose that \ref{H:argmin}, \ref{H:sri} and \ref{lim-la-sum-err-algo} are satisfied. Then, if either \ref{big-la-algo} or \ref{var-ga-algo} is satisfied, $\pa{x_t}_{t \in \N}$ defined in \eqref{seq-w-err} converges weakly towards a minimizer of \eqref{minFnG}. Moreover, if \ref{var-ga-algo} is satisfied and $F$ is uniformly convex, the convergence is strong to the unique global minimizer of \eqref{minFnG}.
\label{gfb-algo-thm}
\end{theorem}

This theorem will be proved after casting it in the more general framework of monotone operator splitting in \sref{proof}.

\begin{remark}
The sufficient condition of strong convergence in \tref{gfb-algo-thm} can be weakened, and other ones can be stated as well. Indeed, the generalized forward-backward algorithm has a structure that bears similarities with the classical forward-backward, since it consists of an explicit forward step, followed by an implicit step where the proximity operators are computed in parallel. In fact, it turns out that the backward step involves a firmly non-expansive operator (see next section), and therefore statements of \cite[Theorem~3.4(iv) and Proposition 3.6]{CombettesWajs05} can be transposed with some care to our algorithm.
\end{remark} 

The formulation of \aref{gfb-algo} is general, but it can be simplified for practical purposes. In particular, the auxiliary variables $z_i$ can all be initialized to $0$, the weights $\om_i$ set equally to $1/n$, and for simplicity the relaxation parameters $\lambda_t$ and the gradient-descent step-size $\ga_t$ can be set constant along iterations. This is typically what has been done in the numerical experiments. 

\section{Monotone Operators and Inclusions}
\label{preliminaries}

The subdifferential of a function in $\Ga_0(\Hh)$ is the best-known example of maximal monotone operator. Therefore, it is natural to extend the generalized forward-backward, \aref{gfb-algo}, to find the zeros of the sum of maximal monotone operators, \ie solve the monotone inclusion \eqref{cns} when the subdifferential is replaced by any maximal monotone operator. This is the goal pursued in \sref{proof} where we provide the proof of a general convergence and robustness theorem whose byproduct is a convergence proof of \tref{gfb-algo-thm}. 

We first begin by recalling some essential definitions and properties of monotone operators that are necessary to our exposition. The interested reader may refer to \cite{Phelps93,BauschkeCombettes11} for a comprehensive treatment.

\subsection{Definitions and Properties}
\label{subsec-defn-prop}

In the following, $A : \Hh \to 2^\Hh$ is a set-valued operator, and $\Id$ is the identity operator on $\Hh$. $A$ is single-valued if the cardinality of $Ax$ is at most 1.

\begin{definition}[Graph, inverse, domain, range and zeros]
\label{def:setvalued}
The \textit{graph} of $A$ is the set $\gra(A) \eqdef \setdef{(x,y) \in \Hh^2}{y \in Ax}$. The \textit{inverse} of $A$ is the operator whose graph is $\gra(\inv{A}) \eqdef \setdef{(x,y) \in \Hh^2}{(y,x) \in \gra(A)}$. The \textit{domain} of $A$ is $\dom(A) \eqdef \setdef{x \in \Hh}{Ax \neq \emptyset}$. The \textit{range} of $A$ is $\ran(A) \eqdef \setdef{y \in \Hh}{\exists x \in \Hh : y \in Ax}$, and its \textit{zeros} set is $\zer(A)\eqdef\setdef{x \in \Hh}{0 \in Ax} = \inv{A} \pa{0}$.
\end{definition}

\begin{definition}[Resolvant and reflection operators]
The \textit{resolvant} of $A$ is the operator $J_A \eqdef \inv{\bpa{\Id+A}}$. The \textit{reflection operator} associated to $J_A$ is the operator $R_A \eqdef 2 J_A - \Id$.
\end{definition}

\begin{definition}[Maximal monotone operator]
$A$ is \textit{monotone} if 
\[
\foralls x,y \in \Hh, u \in Ax \sobjs{and} v \in Ay \Rightarrow \dprod{u - v}{x - y} \geq 0 ~.
\]
It is moreover \textit{maximal} if its graph is not strictly contained in the graph of any other monotone operator.
\end{definition}

\begin{definition}[Non-expansive and $\al$-averaged operators]
$A$ is \textit{non-expan\-sive} if 
\[ 
\foralls x,y \in \Hh, u \in Ax \sobjs{and} v \in Ay \Rightarrow \norm{u - v} \leq \norm{x - y} ~.
\]
For $\al \in ]0,1[$, $A$ is \textit{$\al$-averaged} if there exists $R$ non-expansive such that $A = (1-\al) \Id + \al R$. We denote $\Aa(\al)$ the class of $\al$-averaged operators on $\Hh$. In particular, $\Ademi$ is the class of \textit{firmly non-expansive} operators.
\end{definition}

Note that non-expansive operators are necessarily single-valued and 1-Lipschitz continuous, and so are $\al$-averaged operators since they are also non-expansive. The following lemma gives some useful characterizations of firmly non-expansive operators.

\begin{lemma} Let $A: \dom \pa{A} = \Hh \to \Hh$. The following statements are equivalent:
\begin{enumerate}[label={\rm (\roman{*})}, ref={\rm (\roman{*})}]
\item{\label{fne}} $A$ is firmly non-expansive;
\item{\label{ref}} $2A-\Id$ is non-expansive;
\item{\label{norm-fne}} $\foralls x,y \in \Hh, \, \norm{Ax - Ay}^2 \leq \dprod{Ax - Ay}{x - y}$;
\item{\label{resolv-fne}} $A$ is the resolvent of a maximal monotone operator $A\prim$, \ie $A = J_{A\prim}$.
\end{enumerate} 
\label{lem-fne}
\end{lemma}

\begin{proof}
\ref{fne} $\Leftrightarrow$ \ref{ref}, $A \in \Ademi \Leftrightarrow A = \frac{\Id + R}{2}$ for some $R$ non-expansive.
\ref{fne} $\Leftrightarrow$ \ref{norm-fne}, see \cite{Zarantonello71}.
\ref{fne} $\Leftrightarrow$ \ref{resolv-fne}, see \cite{Minty62}.
\end{proof}

We now summarize some properties of the subdifferential that will be useful in the sequel.

\begin{lemma} Let $F : \Hh \to \R$ be a convex differentiable function, with $1/\be$-Lipschitz continuous gradient, $\be \in ]0,\pinfty[$, and let $G : \Hh \to ]\minfty,\pinfty]$ be a function in $\Ga_0(\Hh)$. Then,
\begin{enumerate}[label={\rm (\roman{*})}, ref={\rm (\roman{*})}]
\item{\label{baillon-haddad}} $\be \nabla F \in \Ademi$, \ie is firmly non-expansive;
\item{\label{subdiff-mm}} $\partial G$ is maximal monotone;
\item{\label{moreau}} The resolvent of $\partial G$ is the proximity operator of $G$, \ie $\prox_{G} = J_{\partial G}$.
\end{enumerate} 
\label{lem-subdiff}
\end{lemma}

\begin{proof}
\ref{baillon-haddad} This is Baillon-Haddad theorem \cite{BaillonHaddad77}.
\ref{subdiff-mm} See \cite{Rockafellar70}.
\ref{moreau} See \cite{Moreau65}.
\end{proof}

We thus consider in the following $n$ maximal monotone operators \linebreak ${A_i : \Hh \to 2^{\Hh}}$ indexed by $i \in \bbket{1,n}$, and a (single-valued) operator $B : \Hh \to \Hh$ and $\be \in ]0,\pinfty[$ such that $\be B \in \Ademi$. Therefore, solving \eqref{cns} can be translated in the more general language of maximal monotone operators as solving the monotone inclusion
\begin{equation}\label{inclusion}
\sobjs{Find} x \in \Hh \sobjs{such that} 0 \in B x + \sum_i A_i x,
\end{equation}
where it is assumed that $\zer \left( B + \sum_i A_i \right) \neq \emptyset$.

\subsection{Product Space}
\label{subsec-defn-prop-other}

The previous definitions being valid for any real Hilbert space, they also apply to the product space $\Hh^n$ endowed with scalar product and norm derived from the ones associated to $\Hh$.  

Let $\pa{\om_i}_{i \in \bbket{1,n}} \in {]0,1[}^n$ such that $\sum_{i=1}^n \om_i = 1$. We consider $\bHh \eqdef \Hh^n$ endowed with the scalar product $\bdprod{\cdot}{\cdot}$, defined as
\begin{equation*}
  \foralls \bx = \pa{x_i}_i, \by = \pa{y_i}_i \in \bHh, \quad \bdprod{\bx}{\by} = \sum_{i=1}^n \om_i \dprod{x_i}{y_i}
\end{equation*}
and with the corresponding norm $\bnorm{\cdot}$. $\bSs \subset \bHh$ denotes the closed convex set defined by $\bSs \eqdef \setdef{ \bx = \pa{x_i}_i \in \bHh}{x_1 = x_2 = \cdots = x_n}$, whose orthogonal complement is the closed linear subspace $\bSs^\bot = \setdef{ \bx = \pa{x_i}_i \in \bHh}{\sum_i \om_i x_i = 0}$. We denote by $\bId$ the identity operator on $\bHh$, and we define the canonical isometry 
\[ \bC: \Hh \to \bSs, x \mapsto (x,\ldots,x) ~. \]
${ \iota_{\bSs} : \bHh \to ]\minfty,\pinfty] }$ and $N_{\bSs}: \bHh \to 2^{\bHh}$ are respectively the indicator function and the normal cone of $\bSs$, that is
\[
\iota_{\bSs}(\bx) \eqdef 
\begin{cases}
 0			& \sobjs{if}  \bx \in \bSs~, \\ 
 \pinfty		& \sobjs{otherwise}~,
 \end{cases} \quad \sobjs{and} \quad
N_{\bSs}(\bx) \eqdef 
\begin{cases}
 \bSs^\bot  & \sobjs{if}  \bx \in \bSs~, \\ 
 \emptyset   & \sobjs{otherwise} ~.
 \end{cases}
\]
Since $\bSs$ is non-empty closed and convex, it is straightforward to see that $N_{\bSs}$ is maximal monotone. To lighten the notation in the sequel, we introduce the following concatenated operators. For every $i \in \bbket{1,n}$, let $A_i$ and $B$ as defined in \eqref{inclusion}. For $\bga = \pa{\ga_i}_{i \in \bbket{1,n}} \in {]0,\pinfty[}^n$, we define ${\bgaA : \bHh \to 2^{\bHh}}, \bx = (x_i)_i \mapsto \bigtimes_{i=1}^n \ga_i A_i(x_i)$, \ie its graph is
\begin{eqnarray*}
\gra \pa{\bgaA} & \eqdef & \bigtimes_{i=1}^n \gra \pa{\ga_i A_i}\\
	      & =      & \setdef{(\bx,\by) \in {\bHh}^2}{\bx = (x_i)_i, \by = (y_i)_i, \sobjs{and} \foralls i, \, y_i \in \ga_i A_i x_i}~,
\end{eqnarray*}
and $\bB : \bHh \to \bHh, \bx = (x_i)_i \mapsto ( B x_i )_i$.

Using the maximal monotonicity of $A_1, \ldots,A_n$ and $B$ it is an easy exercise to establish that $\bgaA$ and $\bB$ are maximal monotone on $\bHh$.

\section{Generalized Forward-Backward Algorithm\\for Monotone Inclusions}
\label{proof}

Now that we have set all necessary material, we are ready to solve the monotone inclusion \eqref{inclusion}. First, we derive an equivalent \textit{fixed point equation} satisfied by any solution of \eqref{inclusion}. From this, we draw an algorithmic scheme and prove its convergence towards a solution, as well as its robustness to errors. Finally, we derive the proof of \tref{gfb-algo-thm}.

\subsection{Fixed Point Equation}

From now on, we denote the set of \textit{fixed points} of an operator $\bT : \bHh \to \bHh$ by $\Fix \bT \eqdef \setdef{\bz \in \bHh}{\bT \bz = \bz}$.

\begin{proposition} Let $\pa{\om_i}_{i \in \bbket{1,n}} \in {]0,1[}^n$. For any $\ga > 0$, $x \in \Hh$ is a solution of \eqref{inclusion} if and only if there exists $(z_i)_{i \in \bbket{1,n}} \in \Hh^n$ such that
\begin{equation}\label{eq-zi}
  \left\{ \begin{array}{L}
	\foralls i, \, z_i = \RgAi ( 2x - z_i - \ga B x ) - \ga B x ~, \\
	x = \sum_i \om_i z_i ~.
\end{array}\right.
\end{equation}
\label{fix-zi}
\end{proposition}

\begin{proof}
set $\ga > 0$, we have the equivalence
\begin{eqnarray*}
 0 \in B x + \sum_i A_i x & \Leftrightarrow & \existss \pa{z_i}_i \in \Hh^n :\, %
\left\{ \begin{array}{l}
	\foralls i, \, \om_i \pa{x - z_i - \ga B x} \in \ga A_i x ~, \\
	x = \sum_i \om_i z_i ~.
\end{array}\right.
\end{eqnarray*}
Now,
\begin{eqnarray*}
\om_i \pa{x - z_i - \ga B x} \in \ga A_i x 
	& \Leftrightarrow & (2x - z_i - \ga B x) - x \in \frac{\ga}{\om_i} A_i x \\
\intext{(by \lref{lem-fne}~\ref{resolv-fne})} & \Leftrightarrow & \hphantom{2}x = \JgAi(2x-z_i-\ga B x) \\
	& \Leftrightarrow & 2x - (2x - z_i) = 2 \JgAi(2x - z_i - \ga B x) \\
	& 		  & \hphantom{2x - (2x - z_i) = } - (2x - z_i - \ga B x) - \ga B x \\
  	& \Leftrightarrow & \hphantom{2}z_i = \RgAi(2x - z_i - \ga B x) - \ga B x ~.
\end{eqnarray*}
\end{proof}

Before formulating a fixed point equation, consider the following preparatory lemma.

\begin{lemma} For all $\bz = (z_i)_i \in \bHh$, $\bb = \pa{b}_i \in \bSs$, and $\bga = \pa{\ga_i}_i \in {]0,\pinfty[}^n$,
\begin{enumerate}[label={\rm (\roman{*})}, ref={\rm (\roman{*})}]
\item{\label{JNS}} $\JNS$ is the \textit{orthogonal projector} on $\bSs$, and $\JNS \bz = \bC\pa{\sum_i \om_i z_i}$;
\item{\label{RNS}} $\RNS \pa{ \bz - \bb } = \RNS \bz - \bb$;
\item{\label{RgA}} $\RgA \bz = \bpa{\RgiAi(z_i)}_i$.
\end{enumerate}
\label{JRSA}
\end{lemma}

\begin{proof}
{~}\\
\ref{JNS}.~From \lref{lem-subdiff} \ref{moreau}, we have for $\bz \in \bHh$, 
\begin{equation*}
\JNS(\bz) = {\argmin_{\by \in \bSs} \bnorm{\bz - \by}} \eqdef \proj_{\bSs}(\bz)~.
\end{equation*}
Now, ${\argmin_{\by \in \bSs} \bnorm{\bz - \by}^2} = \bC\pa{\argmin_{y \in \Hh} \sum_i \om_i \norm{z_i - y}^2}$, where the uni\-que minimizer of $\sum_i \om_i \norm{z_i - y}^2$ is the barycenter of $\pa{z_i}_i$, \ie $\sum_i \om_i z_i$.

\ref{RNS}.~$\JNS$ is obviously linear, and so is $\RNS$. Since $\bb \in \bSs$, $\RNS \bb = \bb$ and the result follows.

\ref{RgA}.~This is a consequence of the separability of $\bgaA$ in terms of the components of $\bz$ implying that $\JgA \bz = (\JgiAi z_i)_i$. The result follows from the definition of $\RgA$.
\end{proof}

\begin{proposition}
$(z_i)_{i \in \bbket{1,n}} \in \Hh^n$ satisfies \eqref{eq-zi} if and only if  $\bz = \pa{z_i}_i$ is a fixed point of the following operator
\begin{equation}\label{operator}
\begin{array}{rcl} \bHh & \longrightarrow & \bHh \\
	\bz & \longmapsto & \frac{1}{2} \big[ \RgA  \RNS + \bId \big]  \big[ \bId - \ga \bBJ \big](\bz) ~,
\end{array}
\end{equation}
with $\bga = \pa{\frac{\ga}{\om_i}}_i$.
\label{fix-z}
\end{proposition}

\begin{proof}
Using \lref{JRSA} in \eqref{eq-zi}, we have $\bC(x) = \JNS \bz$, $\bC(B x) = \bBJ(\bz)$ and $\RNS - \ga \bBJ = \RNS  [\bId - \ga \bBJ]$. Altogether, this yields,
\begin{eqnarray*}
 \bz \sobjs{satisfies} \eqref{eq-zi} & \Leftrightarrow & \hphantom{2}\bz = \RgA \RNS \big[ \bId - \ga \bBJ \big] \bz - \ga \bBJ \bz  \\
  & \Leftrightarrow & 2 \bz = \RgA \RNS \big[ \bId - \ga \bBJ \big] \bz + \big[ \bId - \ga \bBJ \big] \bz  \\
  & \Leftrightarrow & \hphantom{2}\bz = \frac{1}{2} \big[ \RgA \RNS + \bId \big] \big[ \bId - \ga \bBJ \big] \bz
\end{eqnarray*}
\end{proof}

\subsection{Algorithmic Scheme and Convergence}

The expression \eqref{operator} gives us the operator on which is based the generalized forward-backward. We first study the properties of this operator before establishing convergence and robustness results of our algorithm derived from the Krasnoselskij-Mann scheme associated to it.

\begin{proposition}
For all $\bga \in {]0,\pinfty[}^n$, define
\begin{equation}\label{Tb-def}
\Tb : \begin{array}{rcl} \bHh & \longrightarrow & \bHh \\
	\bz & \longmapsto & \frac{1}{2} \left[ \RgA  \RNS + \bId \right] \bz~.
\end{array}
\end{equation}
Then, $\Tb$ is firmly non-expansive, \ie $\Tb \in \Ademi$.
\label{Tb}
\end{proposition}

\begin{proof}
From \lref{lem-fne}, $\RgiAi$ and $\RNS$ are non-expansive. In view of \lref{JRSA} \ref{RgA}, $\RgA$ is non-expansive as well. Finally, as a composition of non-expansive operators, $\RgA\RNS$ is also non-expansive, and the proof is complete by the definition of $\Ademi$.
\end{proof}

\begin{proposition}
For all $\ga \in ]0,2\be[$, define
\begin{equation}\label{Tf-def}
\Tf : \begin{array}{rcl} \bHh & \longrightarrow & \bHh \\
	\bz & \longmapsto & \left[ \bId - \ga \bBJ \right] \bz ~.\end{array}
\end{equation}
Then, $\Tf \in \Aa \pa{ \frac{\ga}{2 \be} }$.
\label{Tf}
\end{proposition}

\begin{proof}
By hypothesis, $\be B \in \Ademi$ and so is $\be\bB$. Then, we have for any $\bx, \by \in \bHh$
\begin{eqnarray}
 \bnorm{\be \bBJ \bx - \be \bBJ \by}^2 & \leq & \bdprod{\be \bBJ \bx - \be \bBJ \by}{\JNS \bx - \JNS \by} \nonumber\\
  & & = \bdprod{\be \JNS \bBJ \bx - \be \JNS \bBJ \by}{\bx - \by} \nonumber\\
  & & = \bdprod{\be \bBJ \bx - \be \bBJ \by}{\bx - \by} ~,\label{monot-cont}
\end{eqnarray} 
where we derive the first equality from the fact that $\JNS$ is self-adjoint (\lref{JRSA} \ref{JNS}), and the second equality using that for all $\bx \in \bSs$, $\bB \bx \subset \bSs$ and $\JNS \bx = \bx$. From \lref{lem-fne} \ref{norm-fne}$\Leftrightarrow$\ref{fne}, we establish that $\be \bBJ \in \Ademi$. We conclude using \cite[Lemma 2.3]{Combettes04}.
\end{proof}

\begin{proposition}
For all $\bga \in {]0,\pinfty[}^n$ and $\ga \in {]0,2\beta[}$, $\Tb\Tf \in \Aa(\al)$, with $\al = \max\pa{\frac{2}{3},\frac{2}{1+2\be/{\ga}}}$.
\label{TbTf}
\end{proposition}

\begin{proof}
As $\Tb$ and $\Tf$ are $\alpha$-averaged operators by \pref{Tb} and \pref{Tf}, it follows from \cite[Lemma 2.2 (iii)]{Combettes04} that their composition is also $\alpha$-averaged with the given value of $\alpha$. 
\end{proof}

\noindent The following proposition defines a maximal monotone operator $\Apga$ which will be useful for caracterizing fixed points of $\Tb\Tf$ as monotone inclusions.

\begin{proposition}\label{Ap-fix-zer}
For all $\bga \in {]0,\pinfty[}^n$ there exists a maximal monotone operator $\Apga$ such that $\Tb = J_{\Apga}$. Moreover for all $\ga > 0$,
\begin{equation}\label{zerApB}
\Fix \Tb \Tf = \zer \pa{\Apga + \ga \bBJ }~.
\end{equation}
\end{proposition}
\begin{proof}
The existence of $\Apga$ is ensured by \pref{Tb} and \lref{lem-fne}~\ref{resolv-fne}. Then for $\bz \in \bHh$,
\begin{eqnarray}
 \bz = \Tb \Tf \bz  & \Leftrightarrow & \bz = \inv{\pa{\bId + \Apga}} \bpa{\bId - \ga \bBJ} \bz \nonumber\\
                    & \Leftrightarrow &  \bz - \ga \bBJ\bz \in \bz + \Apga\bz \nonumber\\
                    & \Leftrightarrow &  0 \in \Apga\bz + \ga \bBJ\bz \nonumber
\end{eqnarray}
\end{proof}

\noindent Now, let us examine the properties of $\Apga$.

\begin{proposition}\label{Ap-incl}
For all $\bga \in \, {]0,\pinfty[}^{n}$ and $\bu,\by \in \bHh$
\begin{equation}\label{caracAp}
\bu \in \Apga \by \Leftrightarrow \pSs{\bu} - \pSp{\by} \in \bgaA \pa{ \pSs{\by} - \pSp{\bu} }~,
\end{equation}
where we denote for $\by \in \bHh$, $\pSs{\by} \eqdef \projS\pa{\by}$ and $\pSp{\by} \eqdef \projP\pa{\by}$.
\end{proposition}
\begin{proof}
First of all, by definition of $\Tb$ we have
\begin{eqnarray}
 \Tb & = & \half \left[ \pa{2\JgA - \bId} \pa{2\JNS - \bId} + \bId \right] \nonumber\\
     & = & \half \left[ 2\JgA (\projS - \projP) - (\projS - \projP) + \projS + \projP \right] \nonumber\\
     & = & \JgA (\projS - \projP) + \projP~. \label{caracTb}
\end{eqnarray} 
By definition we have $\Apga = \inv{\Tb} - \bId$ so that
\begin{eqnarray}
\bu \in \Apga \by & \Leftrightarrow & \bu + \by \in \inv{\Tb} \by \nonumber\\
                  & \Leftrightarrow & \Tb \pa{\bu + \by} = \by \nonumber\\
\intext{(by \eqref{caracTb})} & \Leftrightarrow & \JgA \pa{ \pSs{(\bu+\by)} - \pSp{(\bu+\by)} } = \by - \pSp{(\bu + \by)} \nonumber\\
                  & \Leftrightarrow & \pSs{(\bu+\by)} - \pSp{(\bu+\by)} \in \pSs{\by} - \pSp{\bu} + \bgaA \pa{ \pSs{\by} - \pSp{\bu} } \nonumber\\
                  & \Leftrightarrow & \pSs{\bu} - \pSp{\by} \in \bgaA \pa{ \pSs{\by} - \pSp{\bu} }~. \nonumber
\end{eqnarray}
\end{proof}

We are now ready to state our main result, establishing convergence and robustness of the generalized forward-backward algorithm to solve \eqref{inclusion}.

\begin{theorem}
Let \\
$\pa{\ga_t}_{t \in \N}$ be a sequence in $]0,2\be[$, \\
$\pa{\bga_t}_{t \in \N}$ be a sequence in ${]0,\pinfty[}^n$ such that $\foralls t, \, \bga_t = \pa{\frac{\ga_t}{\om_i}}_i$, \\
$\pa{\la_t}_{t \in \N}$ be a sequence such that $\foralls t, \, \la_t \in I_\la$ (made explicit below),\\
set $\bz_0 \in \bHh$, and for every $t \in \N$, set
\begin{equation}\label{gfb-it}
\bz_{t+1} = \bz_t + \la_t \bpa{\Tbt \bpa{ \Tft \bz_t + \bEps_{2,t} } + \bEps_{1,t} - \bz_t}
\end{equation}
where $\Tbt$ (resp. $\Tft$) is defined in \eqref{Tb-def} (resp. in \eqref{Tf-def}),
and $\bEps_{1,t}, \bEps_{2,t} \in \bHh$.
Set $\varlimsup \ga_t = \bar{\ga}$ and define the following conditions:
\begin{enumerate}[label={\rm (A\arabic{*})}, ref={\rm (A\arabic{*})}]
\setcounter{enumi}{-1}
\item{\label{fix-la-err}}
	\begin{enumerate}[label={\rm (\roman{*})}, ref={\rm (\roman{*})}]
		\item{\label{sol-non-empty}} $\zer\bpa{B+\sum_i A_i} \neq \emptyset$;
		\item{\label{lim-la}} $0 < \varliminf \la_t \leq \varlimsup \la_t < \min\pa{\frac{3}{2},\frac{1+2\be/\bar{\ga}}{2}}$;
		\item{\label{sum-err}} $\sum_{t=0}^{\pinfty} \bnorm{\bEps_{1,t}} < \pinfty$ and $\sum_{t=0}^{\pinfty} \bnorm{\bEps_{2,t}} < \pinfty$.
	\end{enumerate} 
\item{\label{big-la}}
	\begin{enumerate}[label={\rm (\roman{*})}, ref={\rm (\roman{*})}]
		\item{\label{fix-ga}} $\foralls t, \, \ga_t = \bar{\ga} \in ]0,2\be[$;
		\item{\label{bound-la}} $I_\la = \left]0,\min\pa{\frac{3}{2},\frac{1+2\be/\bar{\ga}}{2}}\right[$.
	\end{enumerate} 
\item{\label{var-ga}}
	\begin{enumerate}[label={\rm (\roman{*})}, ref={\rm (\roman{*})}]
		\item{\label{lim-ga}} $0 < \varliminf \ga_t \leq \bar{\ga} < 2\be$;
		\item{\label{small-la}} $I_\la = ]0,1]$.
	\end{enumerate}
\end{enumerate} 
Suppose that \ref{fix-la-err} is satisfied. Then, If either \ref{big-la} or \ref{var-ga} is satisfied, 
\begin{enumerate}[label={\rm (\roman{*})}, ref={\rm (\roman{*})}]
\item{\label{claimTz}} $\bpa{\Tbt \Tft \bz_t - \bz_t}_{t \in \N}$ converges strongly to $0$.
\item{\label{claimz}} $(\bz_t)_{t \in \N}$ converges weakly to a point $\bz \in \Fx \eqdef \bigcap_{t \in \N} \Fix \Tbt \Tft$.
\item{\label{claimx}} $\bpa{x_t \eqdef \sum_i \om_i\bz_{i,t}}_{t \in \N}$ converges weakly to $x \eqdef \sum_i \om_i \bz_i \in \zer{\bpa{B+\sum_i A_i}}$.

Moreover, if \ref{var-ga-algo} is satisfied and $B$ is uniformly monotone, then
\item{\label{claimstrong}} $\pa{x_t}_{t \in \N}$ converges strongly.
\end{enumerate}
\label{gfb-thm}
\end{theorem}

\begin{proof} For sequences in a Hilbert space, strong convergence is denoted by $\cv$ and weak convergence is denoted by $\cvwk$.

\ref{claimTz}-\ref{claimz}.
Suppose first that \ref{fix-la-err} and \ref{big-la} are satisfied.\\
Under \ref{big-la}-\ref{fix-ga}, $\bT \eqdef \Tbb \Tfb$ does not depend on $t$ (stationary operator). For all $t \in \N$, we have
\begin{equation}\label{T-Eps}
  \bz_{t+1} = \bz_t + \la_t \bpa{ \bT \bz_t + \bEps_{t} - \bz_t } ~,
\end{equation}
with $\bEps_{t} \eqdef \Tbb \bpa{ \Tfb \bz_t + \bEps_{2,t} } - \Tbb \bpa{ \Tfb \bz_t } + \bEps_{1,t}$. \pref{Tb} shows that $\Tbb \in \Ademi$ is in particular non-expansive, so that $\bnorm{\bEps_t} \leq \bnorm{\bEps_{2,t}} + \bnorm{\bEps_{1,t}}$, and we deduce from \ref{fix-la-err}-\ref{sum-err} that $\sum_{t=0}^{\pinfty} \bnorm{\bEps_{t}} < \pinfty$. Moreover, by \pref{TbTf} and \ref{big-la}-\ref{fix-ga}, $\bT \in \Aa(\al)$ with $\al = \max\pa{\frac{2}{3},\frac{2}{1+2\be/\bar{\ga}}}$. In particular, $\bT$ is non-expansive and thus $\Fx = \Fix \bT$ is closed and convex. Now, for $t \in \N$, set $\bT_t \eqdef \bId + \la_t \pa{\bT - \Id_{\bHh}}$, the iterations \eqref{T-Eps} can be rewritten
\begin{equation}\label{Tt-Eps}
  \bz_{t+1} = \bT_t \bz_t + \la_t \bEps_{t} ~.
\end{equation}
Since for all $t$, $\al_t \eqdef \lambda_t \al < 1$ by \ref{big-la}-\ref{bound-la}, \cite[Lemma 2.2 (i)]{Combettes04} shows that $\bT_t \in \Aa(\al_t)$, and \eqref{Tt-Eps} is thus a particular instance of \cite[Algorithm 4.1]{Combettes04}. Also, it is clear that for all $t$, $\Fix \bT_t = \Fix \bT$. By \pref{fix-zi} and \pref{fix-z}, \ref{fix-la-err}-\ref{sol-non-empty} provides $\Fx = \bigcap_{t \in \N} \Fix \bT_t \neq \emptyset$. According to \ref{fix-la-err}-\ref{lim-la}, $\varliminf \la_t > 0$ and $\varlimsup \al_t < 1$, so we deduce from \cite[Theorem~3.1 and Remark~3.4]{Combettes04} that
\begin{equation}\label{assymp-reg}
\sum_{t \in \N} \bnormb{\bT_t \bz_t - \bz_t}^2 < \pinfty.
\end{equation}
and that $\pa{\bz_t}_{t}$ is \textit{quasi-Fej\'er monotone} with respect to $\Fx$. By definition of $\bT_t$, \eqref{assymp-reg} gives ${\sum_{t \in \N} {\la_t}^2 \bnormb{\bT \bz_t - \bz_t}^2 < \pinfty}$, which in turn implies $\bT \bz_t - \bz_t \cv 0$ since $\varliminf \la_t > 0$. Then $\bT$ being non-expansive, it follows from the demiclosed principle \cite{Browder67}\cite[Corollary~4.18]{BauschkeCombettes11} that any weak cluster point of $\pa{\bz_t}_t$ belongs to $\Fix \bT$, so that \cite[Theorem~5.5]{BauschkeCombettes11} provides weak convergence towards $\bz \in \Fx$.\\

Suppose now that \ref{fix-la-err} and \ref{var-ga} are satisfied.\\
Again with \pref{fix-zi}, \pref{fix-z} and \ref{fix-la-err}-\ref{sol-non-empty}, $\Fx \neq \emptyset$. From \pref{Tb}, \pref{Tf} and \ref{var-ga}-\ref{lim-ga}, $\Tbt \in \Ademi$ and ${\Tft \in \Aa \pa{ \frac{\ga_t}{2 \be} }}$ for all $t$. So, under assumptions \ref{fix-la-err}-\ref{sum-err} and \ref{var-ga}, \cite[Theorem~3.1 and Remark~3.4]{Combettes04} provides that $\Tbt \Tft \bz_t - \bz_t \cv 0$ (establishing \ref{claimTz}), that for any $\bz \in \Fx$
\begin{equation}\label{Id_Tfz}
 (\bId - \Tft) \bz_t - (\bId - \Tft) \bz \converge{t}{\pinfty} 0~,
\end{equation} 
and that $\pa{\bz_t}_t$ is \textit{quasi-Fej\'er monotone} with respect to $\Fx$. Again, by non-expansivity $\Fx$ is closed and convex, and with \cite[Theorem~5.5]{BauschkeCombettes11}, $\pa{\bz_t}_t$ converges weakly to some point in $\Fx$ if, and only if, all of its weak cluster points lie in $\Fx$.\\
Let thus $\by$ be a weak cluster point of $\pa{\bz_t}_t$. $\pa{\ga_t}_t$ being bounded, we can extract a subsequence $\pa{\bz_{t_\tau}}_\tau$ converging weakly towards $\by$ such that $\pa{\ga_{t_\tau}}_\tau$ converges strongly to some $\gainf$ (${0 < \gainf < 2\be}$ by \ref{var-ga}-\ref{lim-ga}). Fix then $\bz \in \Fx$ and observe that \eqref{Id_Tfz} implies $\bBJ \bz_{t_\tau} \cv \bBJ \bz$.\\
Since $\be \bBJ \in \Ademi$, $\bBJ$ is continuous and monotone, hence maximal monotone \cite[Corollary~20.25]{BauschkeCombettes11}. Consequently, its graph is \textit{sequentially weakly-strongly closed} \cite[Corollary~20.33(ii)]{BauschkeCombettes11}. Because $\bBJ$ is single-valued and $\bz_{t_\tau} \cvwk \by$, we deduce $\bBJ \by = \bBJ \bz$.\\
Now denote for all $t$, $\by_t \eqdef \Tbt \Tft \bz_t$ and $\bu_t \eqdef \pa{\bId - \Tbt} \Tft \bz_t$. It follows from \ref{claimTz} that $\by_t - \bz_t \cv 0$, implying $\by_{t_\tau} \cvwk \by$. Then, $\bu_t = \Tft \bz_t - \Tbt \Tft \bz_t = \bz_t - \ga_t \bBJ \bz_t - \by_t$, so that $\bu_{t_\tau} \longrightarrow - \gainf \bBJ \by$.\\
Moreover, $\bu_t \in \pa{\pa{\bId + \Apgat} - \bId} \Tbt \Tft \bz_t$, hence $\bu_t \in \Apgat \by_t$. Thus for all $t$, $\pSs{\bu_{t}} - \pSp{\by_{t}} \in \bgatA \pa{ \pSs{\by_{t}} - \pSp{\bu_{t}} }$ by \pref{Ap-incl}. If $\pa{\bv, \bu} \in \gra \pa{\bgainfA}$ with ${\bgainf \eqdef \pa{\frac{\gainf}{\om_i}}_i}$, then $\frac{\ga_{t}}{\gainf} \bu \in \bgatA \bv$, and by monotonicity 
\begin{align*}
 \bdprod{\pSs{\bu_{t}} - \pSp{\by_{t}} - \textfrac{\ga_{t}}{\gainf} \bu}{\pSs{\by_{t}} - \pSp{\bu_{t}} - \bv} & \geq 0~, \hspace{3cm} \\
\intertext{by bilinearity and taking into account orthogonality}
 \bdprod{\pSs{\bu_{t}} - \textfrac{\ga_{t}}{\gainf} \bu}{\pSs{\by_{t}} - \pSp{\bu_{t}} - \bv} + \bdprod{\pSp{\by_{t}}}{\pSp{\bu_{t}} + \bv} & \geq 0~. \\
\intertext{By weak convergence, $\pa{\by_{t_\tau}}_\tau$ is bounded. Together with strong convergence of $\pa{\bu_{t_\tau}}_\tau$ and $\pa{\ga_{t_\tau}}_\tau$, \cite[Lemma~2.36]{BauschkeCombettes11} allows to take the limit as $\tau$ tends to infinity in the above inequality. Using $\bBJ \by \in \bSs$,}
 \bdprod{- \gainf \bBJ \by - \bu}{\pSs{\by} - \bv} + \bdprod{\pSp{\by}}{\bv} & \geq 0 \\
 \bdprod{- \gainf \bBJ \by - \pSp{\by} - \bu}{\pSs{\by} - \bv} & \geq 0~.
\end{align*} 
Hence maximality of $\bgainfA$ forces $\pa{ \pSs{\by}, -\gainf \bBJ \by - \pSp{\by} } \in \gra \pa{ \bgainfA }$, \linebreak \ie ${- \gainf \bBJ \by - \pSp{\by} \in \bgainfA \pa{\pSs{\by}}}$. Thus \pref{Ap-incl} provides $- \gainf \bBJ \by \in \Apgainf \by$, and by \pref{Ap-fix-zer}, $\by \in \Fix \Tbinf \Tfinf = \Fx$.\\

\noindent \ref{claimx}.~In both cases, for any $y \in \Hh$, $\dprod{y}{x_t-x} = \dprod{y}{\sum_i \om_i(\bz_{i,t}-\bz_i)} = \sum_i \om_i \dprod{y}{\bz_{i,t} - \bz_i} = \bdprod{\bC(y)}{\bz_t - \bz} \cv 0$ since $\bz_t \cvwk \bz$, hence weak convergence of $(x_t)_{t \in \N}$ towards $x$, which is a zero of $B + \sum_i A_i$ by \pref{fix-zi}.\\

\noindent \ref{claimstrong}.~If $B$ is uniformly monotone, then there exists a non-decreasing function $\phy: [0,\pinfty[ \to [0,\pinfty]$ that vanishes only at 0, such that for all
$x,y \in \Hh$
\begin{equation*}
\dprod{B x - B y}{x - y} \geq \phy(\norm{x - y}) ~.
\end{equation*}
For all $t \in \N$,
\begin{eqnarray*}
\bdprod{\bBJ \bz_t - \bBJ \bz}{\bz_t - \bz} 
	& = & \textsum{i}{} \om_i \dprod{B \pa{ \textsum{i}{} \om_i z_{i,t} } - B \pa{ \textsum{i}{} \om_i z_i}}{z_{i,t} - z_i} \\
	& = & \dprod{B \pa{ \textsum{i}{} \om_i z_{i,t} } - B \pa{ \textsum{i}{} \om_i z_i}}{\textsum{i}{} \om_i (z_{i,t} - z_i)} \\
	& \geq & \phy\pa{\normB{ \textsum{i}{} \om_i (z_{i,t} - z_i)}} = \phy \pa{ \norm{x_t - x} }~.
\end{eqnarray*}
Recall that under \ref{fix-la-err} and \ref{var-ga}, $\bBJ \bz_t \cv \bBJ \bz$ and $\bz_t \cvwk \bz$, so that \linebreak $\bdprod{\bBJ \bz_t - \bBJ \bz}{\bz_t - \bz} \longrightarrow 0$. In view of the properties of $\phy$, we obtain strong convergence of $\pa{x_t}_t$ towards $x$.
\end{proof}

\begin{remark}
\label{rem:fixedga}
In statements \ref{claimTz}-\ref{claimx} of \tref{gfb-thm} under \ref{fix-la-err}-\ref{big-la} (stationary case), assumptions \ref{fix-la-err} can be weakened. More precisely, \ref{fix-la-err}-\ref{lim-la} can be replaced by $\sum_{t \in \N} \la_t(1-\al\la_t) = \pinfty$ where $\alpha=\max(2/3,2/(1+2\beta/\bar{\ga}))$, and \ref{fix-la-err}-\ref{sum-err} by $\sum_{t\in\N} \lambda_t (\bnorm{\bEps_{1,t}} + \bnorm{\bEps_{2,t}}) < \pinfty$. The proof would follow the same lines as \cite[Lemma~5.1]{Combettes04}. Let's note also that a part of assumption \ref{fix-la-err}-\ref{lim-la} on $\varlimsup \la_t$ is not needed under \ref{var-ga}.
\end{remark}

\begin{remark}[Strong Convergence]
\label{rem:strong}
Assumption of uniform monotonicity in the proof of statement \ref{claimstrong} can be relaxed. For instance, the sequence $\pa{\bz_t}_{t \in \N}$ is indeed quasi-Fej\'er monotone with respect to $\Fx$. Thus, if $\inte \Fx \neq \emptyset$, strong convergence occurs by \cite[Lemma~2.8(iv)]{Combettes04}.
\end{remark} 

\begin{corollary} \tref{gfb-algo-thm} holds. \end{corollary}

\begin{proof}
Let $\bpa{z_{i,t}}_{t \in \N}$ and $\bpa{x_t}_{t \in \N}$ be the sequences defined in \eqref{seq-w-err}. Identifying $B$ with $\nabla F$ and $A_i$ with $\partial G_i$ and skipping some calculations, $\pa{ \pa{ z_{i,t} }_i }_{t \in \N}$ follows iterations \eqref{gfb-it} with $\bEps_{1,t} = \pa{\Eps_{1,t,i}}_i$ and $\bEps_{2,t} = \bC \pa{ - \ga_t \Eps_{2,t} }$, providing \ref{fix-la-err}-\ref{lim-la}-\ref{sum-err} in \tref{gfb-thm}. Now, under \ref{H:argmin}-\ref{H:sri}, $\argmin(F + \sum_i G_i) = \zer(\nabla F + \sum_i \partial G_i) \neq \emptyset$, providing \ref{fix-la-err}-\ref{sol-non-empty} in \tref{gfb-thm}. The proof of weak convergence of $\bpa{x_t}_{t \in \N}$ follows from \tref{gfb-thm}-\ref{claimx}. The proof of strong convergence is a consequence of \tref{gfb-thm}-\ref{claimstrong} together with the fact that uniform convexity of a function in $\Ga_0(\Hh)$ implies uniform monotonicity of its subdifferential \cite{BauschkeCombettes11}.
\end{proof}

\section{Discussion\label{discussion}}
\subsection{Special instances}
\label{subsec:spec}
The generalized forward-backward algorithm can be viewed as a hybrid splitting algorithm whose special instances turn out to be classical splitting methods; namely the forward-backward and Douglas-Rachford algorithms.

\paragraph{Relaxed Forward-Backward}
For $n \equiv 1$, we have $\JNS = \Id$, $A \eqdef \bA = A_1$, $\bB = B$ and the operator \eqref{operator} specializes to
\begin{equation}\label{fb}
  \frac{1}{2} \big[ R_{\ga A} + \Id \big] \big[ \Id - \ga B \big] = J_A \bpa{ \Id - \ga B } ~,
\end{equation}
so that $x_t = \bz_t = z_{1,t}$ given by \eqref{gfb-it} (resp. \eqref{seq-w-err} in the optimization case) follows exactly the iterations of the relaxed forward-backward algorithm \cite[Section~6]{Combettes04}, and its convergence properties under assumptions \ref{fix-la-err} and \ref{var-ga}. 

This comparison is of particular interest in the convex optimization case since it may be inspiring to study the convergence rate of the generalized forward-backward on the objective. Indeed, it is now known that the exact forward-backward algorithm enjoys a convergence rate in $O(1/t)$ on the objective \cite{Nesterov07,Bredies08}. Furthermore, there has been several accelerated multistep versions of the exact forward-backward in the literature \cite{Nesterov07,BeckTeboulle09,Tseng08} with a convergence rate of $O(1/t^2)$ on the objective (although no convergence guarantee on the iterate itself is given). Therefore, two possible perspectives of this work would be to investigate the convergence rate (on the objective of course) of the generalized forward-backward and to design a potential multistep acceleration. 

\paragraph{Relaxed Douglas-Rachford}
If we set $B \equiv 0$, the operator \eqref{operator} becomes
\begin{equation}
 \frac{1}{2} \big[ \RgA \RNS + \bId \big] ~.
\end{equation}
Taking $\ga_t = \bar{\ga} \in ]0,\pinfty[, \forall t$, $\bz_t$ provided by \eqref{gfb-it} (resp. \eqref{seq-w-err} in the optimization case) would be equivalent to applying the relaxed Douglas-Rachford algorithm on the product space $\bHh$ for solving $0 \in \sum_i A_i x$ \cite{Spingarn83,CombettesPesquet08}. The convergence statements of \tref{gfb-thm}-\ref{claimTz}-\ref{claimx} holds in this case under \ref{fix-la-err}-\ref{sol-non-empty}, $\la_t \in ]0,2[$ with $\sum_{t \in \N} \la_t(2-\la_t) = \pinfty$ and $\sum_{t\in\N} \lambda_t (\bnorm{\bEps_{1,t}} + \bnorm{\bEps_{2,t}}) < \pinfty$; see Remark~\ref{rem:fixedga} where $\alpha=\frac{1}{2}$ by \pref{Tb}.

\paragraph{Resolvents of the sum of monotone operators}
The generalized for\-ward-back\-ward algorithm provides yet another way for computing the resolvent of the sum of maximal monotone operators at a point $y \in \ran(\Id + \sum_i A_i)$. It is sufficient to take in \eqref{inclusion} $Bx = x - y$ and $\be = 1$. It would be interesting to compare this algorithm with the Douglas-Rachford and Dykstra-based variants \cite{Combettes09b}. This will be left to a future work.

\subsection{Relation to other work}
\label{relation}
\paragraph{Relation to \cite{MonteiroSvaiter10}}
In a finite-dimensional setting, these authors propose an algorithm for the monotone inclusion problem consisting of the sum of a continuous monotone map and a set-valued maximal monotone operator, introducing a ``block-decomposition'' hybrid proximal extragradient (HPE).
They also derive the corresponding convergence rates.

More precisely, our optimization problem can be rewritten in the form considered in \cite[Section 5.3, (51)]{MonteiroSvaiter10}. Indeed, \eqref{minFnG} is equivalent to the linearly constrained convex problem
\begin{equation}\label{HPE-form}
\umin{\bz = (z_i)_i \in \bHh} ~ F\pa{ \scriptstyle{\sum_i} \displaystyle \om_i z_i } + \sum_i G_i(z_i) \qobjq{such that} \projP(\bz) = 0~,
\end{equation}
As $\projP$ is self-adjoint, $\bz$ is an optimal solution if and only if there exists $\bv=(v_i)_i \in \bHh$ such that 
\[
\mathbf{0} \in \pa{ \nabla F(\scriptstyle{\sum_i} \displaystyle \om_i z_i) }_i + \pa{ \partial G_i(z_i)/w_i }_i + \projP(\bv) \qobjq{and} \projP(\bz) = 0 ~,
\]
and the minimizer is given by $x = \sum_i \om_i z_i$.

Let $\varsigma \in ]0,1]$ and $\ga = \varsigma \frac{2 \varsigma \be}{1+\sqrt{1+4\varsigma^2\be^2}}$. Transposed to our setting, their iterations read:

\begin{algorithm}[H]
\label{HPE-iter}
\caption{Iterations Block-Decomposition HPE \cite{MonteiroSvaiter10}.}
\For{$i \in \bbket{1,n}$}{%
		 $\displaystyle z_i \leftarrow \prox_{\frac{\ga}{\om_i} G_i} \bpa{ \ga^2 x + \pa{ 1 - \ga^2 } z_i - \ga \nabla F(x) + \ga \pa{v_i - u} }$;%
	}
\For{$i \in \bbket{1,n}$}{%
		 $\displaystyle v_i \leftarrow v_i - \ga z_i + \ga x$;%
	}
$x \leftarrow \sum_i \om_i z_i$;\\
$u \leftarrow \sum_i \om_i v_i$.
\end{algorithm}

\noindent The update of the $z_i$'s in this iteration shares similarities with the one in Algorithm~\ref{gfb-algo}, where $\ga$ is identified with $\gamma_t$. Nonetheless, the two algorithms are different in some important ways. Our algorithm is robust to errors while there is no proof of such robustness for HPE. Furthermore, HPE carries additional (dual) variables hence increasing the computational load of the algorithm. Finally, unlike our algorithm, the step-size in HPE $\ga$ cannot be iteration-varying, and $\ga < \varsigma$ whatever the Lipschitz constant of $\nabla F$, which is a stronger condition than ours. The latter can have important practical impact.
 
\paragraph{Relation to \cite{CombettesPesquet11}}
While this paper was being released, these authors independently developed another algorithm to solve a class of problems that covers \eqref{inclusion}. They rely on the classical Kuhn-Tucker theory and propose a primal-dual splitting algorithm for solving monotone inclusions involving a mixture of sums, linear compositions, and parallel sums (inf-convolution in convex optimization) of set-valued and Lipschitz operators. More precisely, the authors exploit the fact that the primal and dual problems have a similar structure, cast the problem as finding a zero of the sum of a Lipschitz continuous monotone map with a maximal monotone operator whose resolvent is easily computable. They solve the corresponding monotone inclusion using an inexact version of Tseng's forward-backward-forward splitting algorithm \cite{Tseng00}.

Removing the parallel sum and taking the linear operators as the identity in \cite[(1.1)]{CombettesPesquet11}, one recovers problem \eqref{inclusion}. For the sake of simplicity and space saving we do not reproduce here in full their algorithm. However, adapted to the optimization problem $\min_{x} F(x)+\sum_i G_i(L_ix)$, where each $L_i$ is a bounded linear operator, their algorithm reads ($\adj{G_i}$ is the \textit{Legendre-Fenchel conjugate} of $G_i$):

\begin{algorithm}[H]
\caption{Iterations Primal-Dual Combettes-Pesquet \cite{CombettesPesquet11}.
\label{CoPe-iter}}
$y \leftarrow x - \ga_t \pa{\nabla F(x) + \sum_{i=1}^n \adj{L_i}v_i}$\\
\For{$i \in \bbket{1,n}$}{
		 $\displaystyle z_i \leftarrow v_i + \ga_t L_ix$;\\
		 $\displaystyle v_i \leftarrow v_i - z_i + \prox_{\ga_t \adj{G_i}}(z_i) + \ga_t L_iy$;
	}
$x \leftarrow x - \ga_t \pa{ \nabla F \pa{y} + \sum_{i=1}^n \adj{L_i} \bpa{ \prox_{\ga_t \adj{G_i}}(z_i) } }$;
\end{algorithm}

\noindent Recall that the proximity operator of $\adj{G_i}$ can be easily deduced from that of $G_i$ using Moreau's identity. Taking $L_i=\Id$ in Algorithm~\ref{CoPe-iter} solves \eqref{minFnG}. Similarly to the the generalized forward-backward, this algorithm allows for inexact computations of the involved operators and for varying step-size $\ga_t$. However, if $\ell \eqdef 1/\be$ denotes the Lipschitz constant of $F$, the bound on our step-size sequence is $2/\ell$ while theirs is $1/(\ell + \sqrt{n})$, at least twice lower and degrading as $n$ increases. While we solve
the primal problem, their algorithm solves both the primal and dual ones, which at least doubles the number of auxiliary variables required. Moreover, it also requires two calls to the gradient of $F$ per iteration. Nonetheless, their algorithm is able to solve a more general class of problems.

Finally, let us notice that if one want to use the composition with linear operators, each iteration requires two calls to each one of them and two calls to their adjoints, what can be computationally more expensive than computing directly the proximity operators of the $G_i \circ L_i$'s (see \sref{numeric}).

It is also noteworthy to point out that Tseng's forward-backward-forward algorithm they used is a special case of the HPE method whose iteration complexity results were derived in \cite{MonteiroSvaiter10b}.

\section{Numerical experiments}
\label{numeric}

This section applies the generalized for\-ward-back\-ward to image processing problems. The problems are selected so that other splitting algorithms can be applied as well and compared fairly. In the following, $\Id$ denotes the identity operator on the appropriate space to be understood from the context, $N$ is a positive integer and $\RNN \equiv \R^{N \times N}$ is the set of images of size $N \times N$ pixels.

\subsection{Variational Image Restoration}

We consider a class of inverse problem regularizations, where one wants to recover an (unknown) high resolution image $y_0 \in \RNN$ from noisy low resolution observations $y = \Phi y_0 + w \in \RNN$. We report results using several ill-posed linear operators $\Phi : \RNN \to \RNN$, and focus our attention to convolution and masking operator, and a combination of these operators. In the numerical experiments, the noise vector $w \in \RNN$ is a realization of an additive white Gaussian noise of variance$\sig_w^2$.

The restored image $\hat{y_0}=W\hat{x}$ is obtained by optimizing the coefficients $\hat{x} \in \Hh$ in a redundant wavelet frame~\cite{Mallat99}, where $W : \Hh \to \RNN$ is the wavelet synthesis operator. The wavelet atoms are normalized so that $W$ is a Parseval tight frame, \ie it satisfies $W\adj{W} = \Id$. In this setting, the coefficients are vectors $x \in \Hh\equiv \RNJ$ where the redundancy $J=3 J_0 + 1$ depends on the number of scales $J_0$ of the wavelet transform.

The general variational problem for the recovery reads
\begin{equation}\label{VarPb}
  \umin{x \in \Hh} \{\Psi(x) \eqdef \frac{1}{2} \norm{y - \Phi W x}^2 + \musp \normbk{x} + \mutv \normtv{W x}\} ~.
\end{equation}
The first term in the summand is the \textit{data-fidelity} term, which is taken to be a squared $\ell_2$-norm to reflect the additive white Gaussianity of the noise. The second and third terms are \textit{regularizations}, enforcing priors assumed to be satisfied by the original image. The first regularization is a $\ell_1/\ell_2$-\textit{norm by blocks}, inducing structured sparsity on the solution. The second regularization is a discrete \textit{total variation semi-norm}, inducing sparsity on the gradient of the restored image. The scalars $\musp$ and $\mutv$ are weights -- so-called regularization parameters -- to balance between each terms of the energy $\Psi$. We now detail the properties of each of these three terms.

\subsubsection{Data-Fidelity $\frac{1}{2} \norm{y - \Phi W x}^2$\label{data-f}}

For the inpainting inverse problem, one considers a masking operator 
\begin{equation*}
  \pa{\Mr \, y}_p \eqdef \begin{cases}
	0    & \sobjs{if} p \in \Om , \\
        y_p  & \sobjs{otherwise.}   \end{cases}
\end{equation*}
Where $\Om$ is a set of pixels, taking into account missing or defective sensors that deteriorate the observations; we will denote $\rho = \abs{\Om}/N^2$ the ratio of missing pixels. For the deblurring inverse problem, we consider a convolution with a discrete Gaussian filter of width $\sig$, $\Ks : y \mapsto g_{\sig} \ast y$, normalized to a unit mass. This simulates a defocus effect or low-resolution sensors. In the following, we thus consider $\Phi$ being equal either to $\Mr$, $\Ks$ or the composition $\Mr  \Ks$.

\medskip

Denoting $L \eqdef \Phi W$, the fidelity term thus reads $F(x) = \frac{1}{2} \norm{y - L x}^2$. The function $F$ corresponds to the smooth term in \eqref{minFnG}. Its gra\-dient $\nabla F : x \mapsto \adj{L} \pa{Lx - y}$ is Lipschitz-continuous with constant $\beta^{-1} \leq \norm{\Phi W}^2 = 1$.

For any $\ga > 0$, the proximity operator of $F$ reads 
\begin{equation}\label{proxF}
 \prox_{\ga F}(x) = \inv{\pa{\Id+\ga \adj{L}L}}\pa{x+\ga \adj{L}y}.
\end{equation}
The vector $\adj{L}y$ can be precomputed, but inverting $\Id+\ga \adj{L}L$ may be problematic. For $L \equiv \Id$, this is trivial. For inpainting or deblurring alone, as $\Phi$ is associated to a Parseval tight frame, $L \equiv \Mr W$ or $L \equiv \Ks W$, the Sherman-Morrison-Woodbury formula gives
\begin{align}
  \inv{\pa{\Id+\ga \adj{L}L}} = & \Id - \adj{L} \inv{\pa{\Id+\ga L\adj{L}}} L \nonumber \\
			      = & \Id - \adj{W} \adj{\Phi} \inv{\pa{\Id+\ga \Phi \adj{\Phi} }} \Phi W \label{proxPhiW} ~.
\end{align}
Since $\Mr$ (resp. $\Ks$) is a diagonal operator in the pixel domain (resp. Fourier domain), \eqref{proxPhiW} can be computed in $O(N^2)$ (resp. $O(N^2 \log N)$) operations. However, the composite case $L \equiv \Mr \Ks W$ is more involved. An auxiliary variable is required, replacing $F : \Hh \to \R$ by ${\tilde{F} : \Hh \times \RNN \to ]\minfty,\pinfty]}$ defined by
\begin{equation}\label{tildeF}
  \tilde{F}(x,u) = \frac{1}{2} \norm{y - \Mr u}^2 + \iota_{C_{\Ks W}}(x,u) = G_1(x,u) + G_2(x,u) ~,
\end{equation}
where $C_{\Ks W} \eqdef \setdef{\pa{x,u} \in \Hh \times \RNN}{u = \Ks W x}$. Only then, $\prox_{\ga G_1}$ can be computed from \eqref{proxF}, and $\prox_{\ga G_2}$ is the orthogonal projection on \linebreak ${\ker([\Id, -\Ks W])}$ \cite{Dupe11b,Briceno-Arias10}, which involves a similar inversion as in \eqref{proxPhiW}.

\subsubsection{Regularization $\musp \normbk{x}$\label{sparse-reg}}

Sparsity-promoting regularizations over wavelet (and beyond) coefficients are popular to solve a wide range of inverse problems~\cite{Mallat99}. Figure \ref{noblock}, left, shows an example of orthogonal wavelet coefficients of a natural image, where most of the coefficients have small amplitude, they are thus quite sparse. A way to enforce this sparsity is to use the $\ell_1$-norm of the coefficients $\norm{x}_1 = \sum_p |x_p|$. 

The presence of edges or textures creates structured local dependencies in the wavelet coefficients of natural images. A way to take into account those dependencies is to replace the absolute value of the coefficients in the $\ell_1$-norm by the $\ell_2$-norm of groups (or \textit{blocks}) of coefficients \cite{yuan-group-lasso}. This is known as the mixed $\ell_1/\ell_2$-norm by
\begin{equation}\label{normbk}
  \normbk{x} = \sum_{\bk \in \Bb} \mu_\bk \norm{x_\bk} = \sum_{\bk \in \Bb} \mu_\bk \sqrt{\sum_{p \in \bk} x_p^2} ~,
\end{equation}
where $p$ indexes the coefficients, the blocks $\bk$ are sets of indexes, the \textit{block-structure} $\Bb$ is a collection of blocks and $x_\bk \eqdef \pa{x_p}_{p \in \bk}$ is a subvector of $x$. The positive scalars $\mu_\bk$ are weights tuning the influence of each block. It is a norm on $\Hh$ as soon as $\Bb$ covers the whole space, \ie $\foralls p \in \bbket{1,N}^2 \times \bbket{1,J}, \, \exists \bk \in \Bb \sobjs{s.t.} p \in \Bb$. Note that for $\Bb \equiv \bigcup_p \{p\}$ and $\mu_{ \ket{p} } \equiv 1$ for all $p$, it reduces to the $\ell_1$-norm.

We mentionned in the introduction that the proximal operator of a $\ell_1$-norm is a soft-thresholding on the coefficients. Similarly, it is easy to show that whenever $\Bb$ is \textit{non-overlapping}, \ie ${\foralls \bk,\bk\prim \in \Bb}, \, {\bk \cap \bk\prim = \emptyset}$, the proximity operator of $\normbk{\cdot}$ is a soft-thresholding by block
\begin{equation*}
  \prox_{\mu \normbk{\cdot}}\bpa{\pa{x_\bk}_\bk} = \bpa{\Th_{\mu_\bk \cdot \mu}(x_\bk)}_\bk ~,
\end{equation*}
with
\begin{equation*}
\Th_{\mu}(x_\bk) = 
\begin{cases}
0                                       & \sobjs{if} \norm{x_\bk} < \mu~, \\ 
\pa{ 1-\frac{\mu}{\norm{x_\bk}} } x_\bk & \sobjs{otherwise~,}
\end{cases}
\end{equation*}
and the coefficients $x_p$ not covered by $\Bb$ remaining unaltered.

Non-overlapping block structures break the \textit{translation invariance} that is underlying most traditional image models. To restore this invariance, one can consider overlapping blocks, as illustrated in \cff{blocks2}. Computing $\prox_{\normbk{\cdot}}$ in this case is not as simple as for the non-overlapping case, because the blocks cannot be treated separately. For tree-structured blocks (\ie $\bk \cap \bk\prim \neq \emptyset \Rightarrow \bk \subset \bk\prim \sobjs{or} \bk\prim \subset \bk$), \cite{Jenatton10} proposes a method involving the computation of a min-cost flow. This could be computationally expensive and do not address the general case anyway. Instead, it is always possible to decompose the block structure as a finite union of non-overlapping sub-structures $\Bb = \bigcup_{i} \Bb_i$. The resulting term can finally be split into $\normbk{x} = \sum_{\bk \in \Bb} \norm{x_\bk} = \sum_i \sum_{\bk \in \Bb_i} \norm{x_\bk} = \sum_i \normbki{x}$, where each $\normbki{\cdot}$ is simple.

\begin{figure}[!ht]
\bigcenter
\subfigure[$\norm{x}_1 = \sum_p \abs{x_p}$]{
\includegraphics[width=4cm]{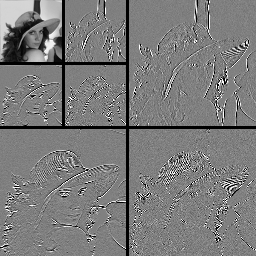}
\label{noblock}}
\subfigure[$\normbk{x} = \sum_{\color{red} \bk \in \Bb} \norm{x_{\color{red} \bk}}$]{
\includegraphics[width=4cm]{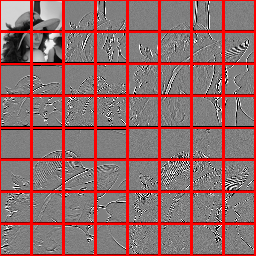}
\label{blocks1}}
\subfigure[$\normbk{x}%
= \norm{x}_{1,2}^{\textcolor{red}{\Bb_1}} + \norm{x}_{1,2}^{\textcolor{blue}{\Bb_2}}%
$]{
\includegraphics[width=4cm]{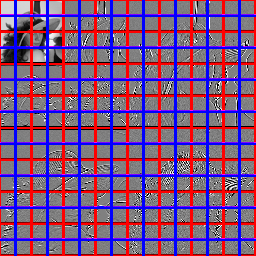}
\label{blocks2} }
\caption{Illustration of the block $\ell_1/\ell_2$-norm. \subref{noblock} sparsity of the image in an orthogonal wavelet decomposition (gray pixels corresponds to low coefficients); \subref{blocks1} a non-overlapping block structure; \subref{blocks2} splitting of a more complex structure into two non-overlapping layers.}
\label{blocks}
\end{figure}

In our numerical experiments where $\Hh \equiv \RNJ$, coefficients within each resolution level (from $1$ to $J$) and each subband are grouped according to all possible square spatial blocks of size $S \times S$; which can be decomposed into $S^2$ non-overlapping block structures.

\subsubsection{Regularization $\mutv \normtv{W x}$\label{tv-reg}}

The second regularization favors piecewise-smooth images, by inducing sparsity on its gradient \cite{ROF92}. The total variation semi-norm can be viewed as a specific instance of $\ell_1/\ell_2$-norm, $\normtv{y} = \normbktv{\gradient y}$, with
\begin{equation*}
 \gradient : \left\{ \begin{array}{rcc}
          \RNN & \longrightarrow &  \RG\\
          y    & \longmapsto & (V \ast y, H \ast y)
          \end{array} \right.%
\qobjq{and} \normbktv{\pa{v,h}} = \sum_{p \in \bbket{1,N}^2} \sqrt{{v_p}^2 + {h_p}^2} ~,
\end{equation*}
where the image gradient is computed by finite differences through convolution with a vertical filter $V$ and a horizontal filter $H$, and $\Bb_{\mathrm{TV}}$ is clearly non-overlapping. For some special gradient filters, the modified TV semi-norm can be splitted into simple functions, see for instance \cite{CombettesPesquet08,Pustelnik11}. However, we consider more conventional filters
\begin{equation*}
V =\left( \begin{array}{cc}  -1 & 0 \\ 1 & 0 \end{array} \right)%
\qobjq{and}%
H =\left( \begin{array}{cc}  -1 & 1 \\ 0 & 0 \end{array} \right)
\end{equation*}
centered in the upper-left corner. Introducing an auxiliary variable as ad\-vo\-ca\-ted in \eqref{tildeF}, the main difficulty remains to invert the operator $\pa{\Id + \ga \gradient \adj{\gradient}}$. Under appropriate boundary conditions, this can be done in the Fourier domain in $O(N^2 \log(N))$ operations.

\subsection{Resolution with Splitting Methods}

\subsubsection{Tested Algorithms}

We now give the details of the different splitting strategies required to apply the three tested algorithms to \eqref{VarPb}.

\paragraph{Generalized Forward-Backward (\GFB)} The problem is rewritten under the form \eqref{minFnG} as
\begin{equation}\label{GFB-form}
 \umin{\substack{x \in \Hh\\u \in \RG}} \frac{1}{2} \norm{y - \Mr \Ks W x}^2 + \musp \sum_{i=1}^{S^2} \normbki{x} + \mutv \normbktv{u} + \iota_{C_{\gradient W}}(x,u) ~,
\end{equation}
with $F(x) \equiv \frac{1}{2} \norm{y - \Mr \Ks W x}^2$ and $n \equiv S^2 + 2$. The indicator function $\iota_{C_{\gradient W}}$ is defined similarly as in \eqref{tildeF}. In \aref{gfb-algo}, we set equal weights $\om_i \equiv 1/n$, a constant gradient step-size $\ga \equiv 1.8 \be$ and a constant relaxation parameter to $\la \equiv 1$.

\paragraph{Relaxed Douglas-Rachford (\DR)} Here the problem is split as
\begin{equation*}
 \umin{\substack{x \in \Hh\\u_1 \in \RNN\\u_2 \in \RG}} \frac{1}{2} \norm{y - \Mr \, u_1}^2 + \iota_{C_{\Ks W}}(x,u_1) + \musp \sum_{i=1}^{S^2} \normbki{x} + \mutv \normbktv{u_2} + \iota_{C_{\gradient W}}(x,u_2) ~,
\end{equation*}
and solved with \aref{gfb-algo}, where $F \equiv 0$ and $n \equiv S^2 + 4$. As mentioned in \sref{discussion}, this corresponds to a relaxed version of the Douglas-Rachford algorithm, with best results when $\ga \equiv 1/n$.

\paragraph{Primal-Dual Chambolle-Pock (\ChPo)} A way to avoid operator inversions is to rewrite the original problem as
\begin{equation*}
  \umin{x \in \Hh} G (\La x)
\end{equation*}
where
\begin{equation*}
 \La : \left\{ \begin{array}{ccc} \Hh & \longrightarrow & \RNN \times \pa{\Hh}^{S^2} \times \RG\\
     x & \longmapsto & \bpa{ \Mr \Ks W x, x, \dots, x, \gradient W x} \end{array} \right.~,
\end{equation*}
and
\begin{equation*}
G : \left\{ \begin{array}{rcl} \RNN \times \pa{\Hh}^{S^2} \times \RG & \longrightarrow & \R \\
    \bpa{ u_1, x_1, \dots, x_{S^2}, g} & \longmapsto & \frac{1}{2} \norm{y - u_1}^2 + \musp \sum_{i=1}^{S^2} \norm{x_i}_{1,2}^{\Bb_i} + \mutv \normbktv{g} \end{array} \right.~.
\end{equation*}
The operator $\La$ is a concatenation of linear operators and its adjoint is easy to compute, and $G$ is simple, being a separable mixture of simple functions. Note that this is not the only splitting possible. For instance, one can write the problem on a product space as $\umin{(x_i)_i \in \bHh} \iota_{\bSs}(\pa{x_i}_i) + \sum_i G_i(\La_i x_i)$, where $G_i$ is each of the functions in $G$ above, and $\La_i$ is each of the linear operators in $\La$.

To solve this, we here use the primal-dual relaxed Arrow-Hurwicz algorithm described in \cite{ChambollePock11}. According to the notations in that paper, we set the parameters $\sig \equiv 1$, $\tau \equiv \frac{0.9}{\sig(1+S^2+8)}$ and $\th \equiv 1$.

\paragraph{Block-Decomposition Hybrid Proximal Extragradient (\HPE)} We split the problem written in \eqref{GFB-form} according to \eqref{HPE-form}, and set equals weights $w_i \equiv 1/n$. According to \sref{relation}, we set the parameter $\varsigma \equiv 0.9$.

\paragraph{Primal-Dual Combettes-Pesquet (\CoPe)} Finally, the problem takes its simplest form
\begin{equation}\label{CoPe-form}
 \umin{x \in \Hh} \frac{1}{2} \norm{y - \Mr \Ks W x}^2 + \musp \sum_{i=1}^{S^2} \normbki{x} + \mutv \normbktv{\gradient W x}~.
\end{equation}
As long as $\mutv \equiv 0$ (no TV-regularization), this is exactly \eqref{GFB-form}; we apply \aref{CoPe-iter} where $L_i \equiv \Id$ for all $i$ and $\ga \equiv 0.9/(1+S)$. However with TV-regularization, we avoid the introduction of the auxiliary variable $u$ with $L_{S^2+1} \equiv \gradient W$ and $\ga \equiv 0.9/(1+\sqrt{S^2 + 8})$.

\subsubsection{Results}

All experiments were performed on a discrete image of width $N \equiv 256$, with values in the range $[0,1]$. The additive white Gaussian noise has standard-deviation $\sig_w \equiv 2.5 \cdot 10^{-2}$. The reconstruction operator $W$ uses non-separable, bi-dimensional Daubechies wavelets with 2 vanishing moments. It is implemented such that each atom has norm $2^{-j}$, with $j \in \bbket{1,J_0}$ and where $J_0$ is the coarsest resolution level. Accordingly, we set the weights $\mu_\bk$ in the $\ell_1/\ell_2$-norm to $2^{-j}$ at the resolution level $j$ of the coefficients in block $\bk$. We use $J_0 \equiv 4$, resulting in a dictionary with redundancy $J = 3 J_0 +1 = 13$. All algorithms are implemented in \textsc{Matlab}\footnote{An implementation of the generalized forward-backward, as well as the codes and materials for the experiments, are available at \url{http://www.ceremade.dauphine.fr/~raguet/}}.

Results are presented in Figures \ref{deconv}, \ref{inpaint}, \ref{composite} and \ref{composite-tv}. For each problem, the five algorithms were run $1000$ iterations (initialized at zero), while monitoring their objective functional values $\Psi$ along iterations. $\Psi_{\min}$ is fixed as the minimum value reached over the five algorithms (in our experiments, this was always the generalized forward-backward), and evolution of the objectives compared to $\Psi_{\min}$ is displayed for the first $100$ iterations. Because the computational complexity of an iteration may vary between algorithms, computation times for $100$ iterations (no parallel implementation) are given beside the curves. Below the energy decay graph, one can find from left to right the original image, the degraded image and the restored image after $100$ iterations of generalized forward-backward. Degraded and restored images quality are given in term of the signal-to-noise ratio ($\ins{SNR}$).
 
\paragraph{Comparison to algorithms that do not use the (gradient) explicit step (\ChPo, \DR)} For the first three experiments, there is no total variation regularization. In the deblurring task (\cff{deconv}), blocks of size $2 \times 2$ are used. \GFB is slightly better than the others and iteration cost of \ChPo is too high for this problem. When increasing the number of block structures (inpainting, \cff{inpaint}, size $4 \times 4$) computation times tends to be similar but \GFB clearly outperforms the others for the task. However, one advantage of using the gradient becomes obvious in the composite case (\ie $\Phi \equiv \Mr \Ks$): in \cff{composite}, \DR performs hardly better than \ChPo. Indeed, in contrast to previous cases (see \sref{data-f}), $F$ is not simple anymore and the introduction of the auxiliary variable decreases the efficiency of each iteration of \DR. This phenomenon is further illustrated in the last case, where the total variation is added, introducing another auxiliary variable.

\paragraph{Comparison to algorithms that use the (gradient) explicit step (\HPE, \CoPe)} In the first experiment where $n$ is small, the iterations of \HPE and \CoPe are almost as efficient as the iterations of \GFB but take more time to compute, especially for \CoPe that needs twice more calls to $\nabla F$. In the second setting, \HPE and \CoPe are hardly better than \DR, maybe suffering from small gradient step-sizes. They perform better in the composite setting, but require more computional time than \GFB. In the last setting, iterations of \CoPe are still not as efficient as iterations of \GFB in spite of their higher computational load due to the composition by the linear operator $\gradient W$ (see \eqref{CoPe-form}).

\bigskip

Finally, let us note that in the composite case (\ie $\Phi \equiv \Mr \Ks$), the $\ins{SNR}$ of the restored image is greater when using both regularizations rather than one or the other separately. Moreover, we observed that it yields restorations more robust to variations of the parameters $\musp$ and $\mutv$. Those arguments seem to be in favor of mixed regularizations.

\section{Conclusion}

We have introduced in this paper a novel proximal splitting method able to handle convex functionals that are the sum of a smooth term and several simple functions. It generalizes existing schemes by enlarging the class of problems that can be solved efficiently with proximal methods to the case where one of the function is smooth but not simple. We provided theoretical guarantees on the convergence and robustness of the algorithm even for the more general problem of finding the zeros of the sum of maximal monotone operators, one of which is also co-coercive. Numerical experiments on convex optimization problems encountered in inverse problems show evidence of the advantages of our approach for large-scale imaging problems.

In analogy with first-order methods such as the forward-backward algorithm, establishing convergence rates (on the objective) and designing multistep accelerations are possible perspectives that we leave to a future work.

\newpage
\floatstyle{boxed}
\restylefloat{figure}

\begin{figure}[H]
\bigcenter
\subfiguretopcaptrue
\subfigure[$\log(\Psi - \Psi_{\min})$ vs. iteration \#]{
\includegraphics[width=.5\textwidth]{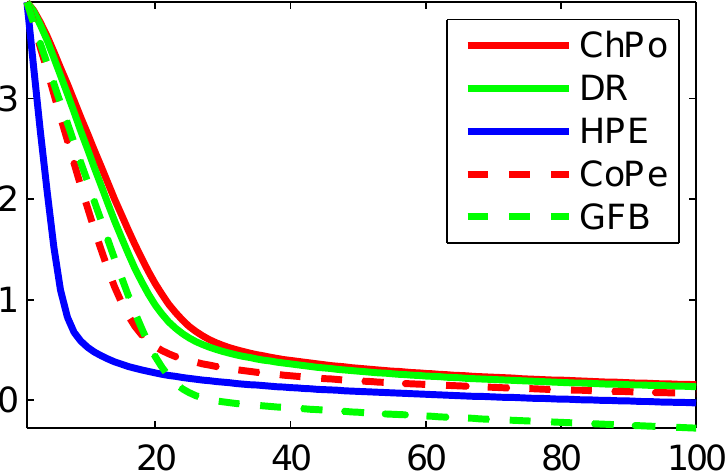}
\label{deconv-curves}}
\subtable[computing time]{
\begin{tabular}{>{\Large}l>{\Large}l}
		   &		     \\
	$t_{\ChPo}$  &$= 153~\ins{s}$\\
	$t_{\DR}$  &$= 95~\ins{s}$\\
	$t_{\HPE}$  &$= 148~\ins{s}$\\
	$t_{\CoPe}$ &$= 235~\ins{s}$\\
	$t_{\GFB}$ &$= 73~\ins{s}$\\
\end{tabular}} \\
\subfiguretopcapfalse
\subfigure[LaBoute $y_0$]{
\includegraphics[width=.3\textwidth]{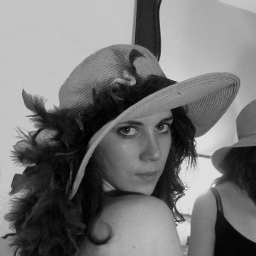}
\label{deconv-img} }
\subfigure[${y = \Ks y_0 + w}, \, 19.63~\ins{dB}$]{
\includegraphics[width=.3\textwidth]{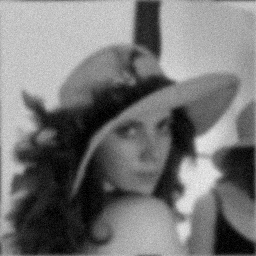}
\label{deconv-obs} }
\subfigure[$\hat{y_0} = W\hat{x}, \, 22.45~\ins{dB}$]{
\includegraphics[width=.3\textwidth]{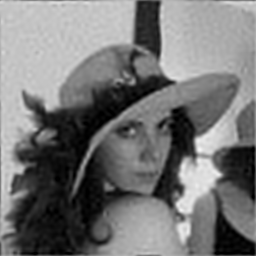}
\label{deconv-recov} }
\caption{Deblurring: $\sig = 2$; $\musp = 1.3 \cdot 10^{-3}$; $S = 2$; $\mutv = 0$.}
\label{deconv}
\end{figure}

\vspace{-.5cm}

\begin{figure}[H]
\bigcenter
\subfiguretopcaptrue
\subfigure[$\log(\Psi - \Psi_{\min})$ vs. iteration \#]{
\includegraphics[width=.5\textwidth]{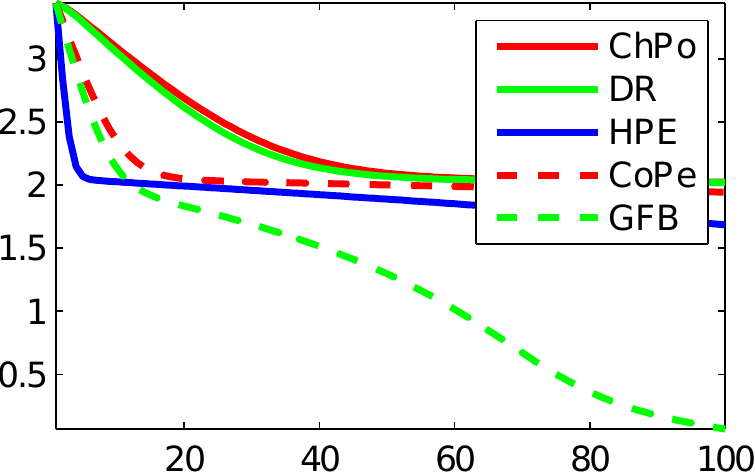}
\label{inpaint-curves}}
\subtable[computing time]{
\begin{tabular}{>{\Large}l>{\Large}l}
		   &		     \\
	$t_{\ChPo}$  &$= 229~\ins{s}$\\
	$t_{\DR}$  &$= 219~\ins{s}$\\
	$t_{\HPE}$  &$= 352~\ins{s}$\\
	$t_{\CoPe}$ &$= 340~\ins{s}$\\
	$t_{\GFB}$ &$= 203~\ins{s}$\\
\end{tabular}} \\
\subfiguretopcapfalse
\subfigure[LaBoute $y_0$]{
\includegraphics[width=.3\textwidth]{images/LaBoute.png}
\label{inpaint-img} }
\subfigure[$y = \Mr y_0 + w, \, 1.54~\ins{dB} $]{
\includegraphics[width=.3\textwidth]{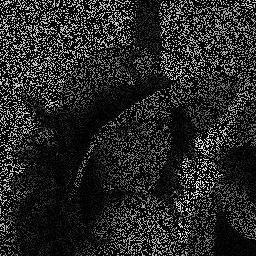}
\label{inpaint-obs} }
\subfigure[$\hat{y_0} = W\hat{x}, \, 21.66~\ins{dB}$]{
\includegraphics[width=.3\textwidth]{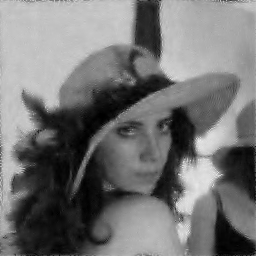}
\label{inpaint-recov} }
\caption{Inpainting: $\rho = 0.7$; $\musp = 2.6 \cdot 10^{-3}$; $S = 4$; $\mutv = 0$.}
\label{inpaint}
\end{figure}

\begin{figure}[H]
\bigcenter
\subfiguretopcaptrue
\subfigure[$\log(\Psi - \Psi_{\min})$ vs. iteration \#]{
\includegraphics[width=.5\textwidth]{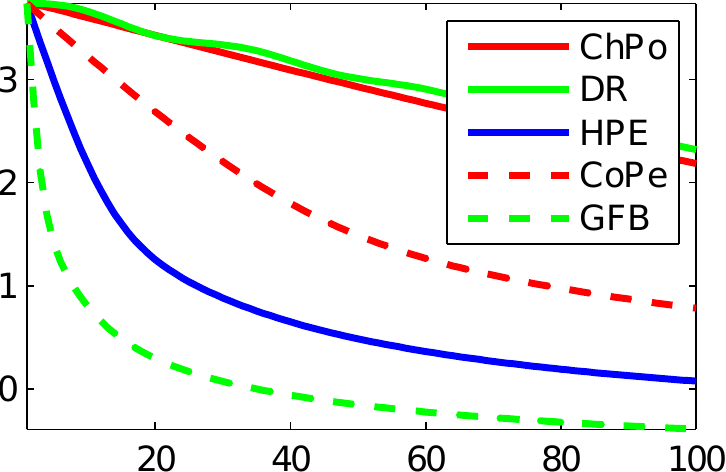}
\label{composite-curves}}
\subtable[computing time]{
\begin{tabular}{>{\Large}l>{\Large}l}
		   &		     \\
	$t_{\ChPo}$  &$= 313~\ins{s}$\\
	$t_{\DR}$  &$= 256~\ins{s}$\\
	$t_{\HPE}$  &$= 342~\ins{s}$\\
	$t_{\CoPe}$ &$= 268~\ins{s}$\\
	$t_{\GFB}$ &$= 233~\ins{s}$\\
\end{tabular}} \\
\subfiguretopcapfalse
\subfigure[LaBoute $y_0$]{
\includegraphics[width=.3\textwidth]{images/LaBoute.png}
\label{composite-img} }
\subfigure[$y = \Mr \Ks y_0 + w, \, 3.93~\ins{dB}$]{
\includegraphics[width=.3\textwidth]{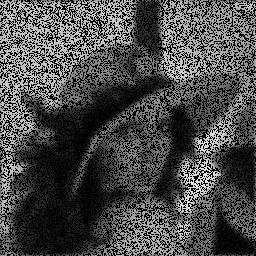}
\label{composite-obs} }
\subfigure[$\hat{y_0} = W\hat{x}, \, 20.77~\ins{dB}$]{
\includegraphics[width=.3\textwidth]{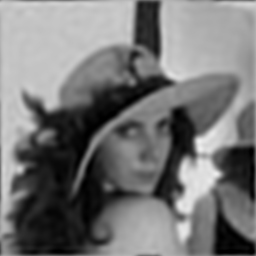}
\label{composite-recov} }
\caption{Composite: $\sig = 2$; $\rho = 0.4$; $\musp = 1.0 \cdot 10^{-3}$; $S = 4$; $\mutv = 0$.}
\label{composite}
\end{figure}

\vspace{-.5cm}

\begin{figure}[H]
\bigcenter
\subfiguretopcaptrue
\subfigure[$\log(\Psi - \Psi_{\min})$ vs. iteration \#]{
\includegraphics[width=.5\textwidth]{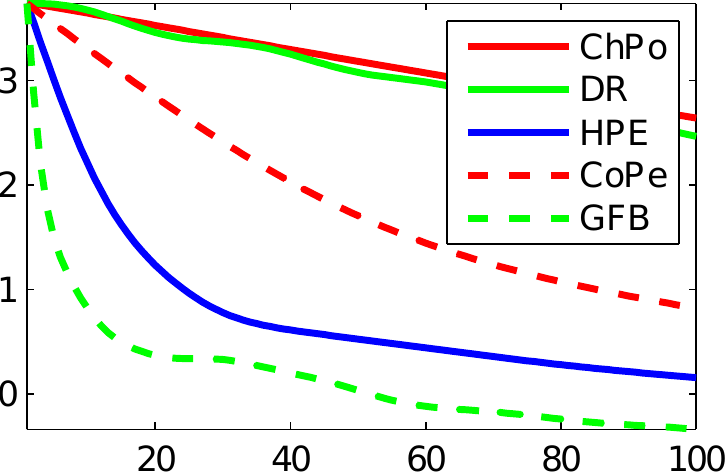}
\label{composite-tv-curves}}
\subtable[computing time]{
\begin{tabular}{>{\Large}l>{\Large}l}
		   &		     \\
	$t_{\ChPo}$  &$= 358~\ins{s}$\\
	$t_{\DR}$  &$= 294~\ins{s}$\\
	$t_{\HPE}$  &$= 409~\ins{s}$\\
	$t_{\CoPe}$ &$= 441~\ins{s}$\\
	$t_{\GFB}$ &$= 286~\ins{s}$\\
\end{tabular}} \\
\subfiguretopcapfalse
\subfigure[LaBoute $y_0$]{
\includegraphics[width=.3\textwidth]{images/LaBoute.png}
\label{composite-tv-img} }
\subfigure[$y = \Mr \Ks y_0 + w, \, 3.93~\ins{dB}$]{
\includegraphics[width=.3\textwidth]{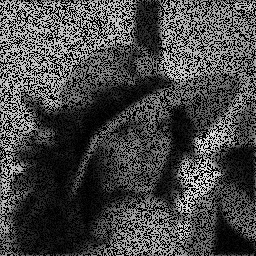}
\label{composite-tv-obs} }
\subfigure[$\hat{y_0} = W\hat{x}, \, 22.48~\ins{dB}$]{
\includegraphics[width=.3\textwidth]{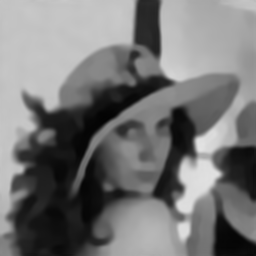}
\label{composite-tv-recov} }
\caption{Composite: $\sig = 2$; $\rho = 0.4$; $\musp = 5.0 \cdot 10^{-4}$; $S = 4$; $\mutv = 5.0 \cdot 10^{-3}$.}
\label{composite-tv}
\end{figure}

\floatstyle{plain}
\restylefloat{figure}


\bibliographystyle{plain}
\bibliography{GFB}

\end{document}